\documentclass{amsart}
\usepackage{amsmath,bbm}
\renewcommand{\epsilon}{\varepsilon}
\renewcommand{\phi}{\varphi}
\renewcommand{\kappa}{\varkappa}
\renewcommand{\tilde}{\widetilde}
\newcommand{\cal}{\mathcal}
\newcommand{\IE}{\mathbb{E}}
\newcommand{\I}{\mathbbm{1}}
\newcommand{\IN}{\mathbb{N}}
\newcommand{\IP}{\mathbb{P}}

\newcommand{\IZ}{\mathbb{Z}}
\newtheorem{theorem}{Theorem}[section]
\newtheorem{lemma}[theorem]{Lemma}
\newtheorem{corollary}[theorem]{Corollary}
\newtheorem{proposition}[theorem]{Proposition}
\newtheorem{remark}[theorem]{Remark}
\numberwithin{equation}{section}
\begin{document}
\title{Limit laws of transient excited random walks on integers}

\author{Elena Kosygina and Thomas Mountford} 
\begin{abstract}
  We consider excited random walks (ERWs) on $\IZ$ with a bounded
  number of i.i.d.\ cookies per site without the non-negativity
  assumption on the drifts induced by the cookies.  Kosygina and
  Zerner \cite{KZ08} have shown that when the total expected drift
  per site, $\delta$, is larger than 1 then ERW is transient to the
  right and, moreover, for $\delta>4$ under the averaged measure it
  obeys the Central Limit Theorem. We show that when $\delta\in(2,4]$
  the limiting behavior of an appropriately centered and scaled
  excited random walk under the averaged measure is described by a
  strictly stable law with parameter $\delta/2$. Our method also
  extends the results obtained by Basdevant and Singh \cite{BS08b}
  for $\delta\in(1,2]$ under the non-negativity assumption to the
  setting which allows both positive and negative cookies.

\medskip

\noindent{\sc R\'esum\'e.}  On consid\`ere des marches al\'eatoires
excit\'ees sur $\IZ$ avec un nombre born\'e de cookies i.i.d.\ \`a
chaque site, ceci sans l'hypoth\`ese de positivit\'e.  Auparavent,
Kosygina et Zerner \cite{KZ08} ont \'etabli que si le drift totale moyenne par
site, $\delta$, est strictement superieur \`a 1, alors la marche est
transiente (vers la droite) et, de plus, pour $\delta > 4$ il y a un
th\'eor\`eme central limite pour la position de la marche.  Ici, on
d\'emontre que pour $\delta \in (2, 4]$ cette position, convenablement
centr\'ee et r\'eduite, converge vers une loi stable de paramètre
$\delta / 2$.  L'approche permet \'egalement d'\'etendre les
r\'esultats de Basdevant et Singh \cite{BS08b} pour $\delta \in (1,2]$ \`a notre
cadre plus g\'en\'eral.
\end{abstract}
\maketitle

%%%%%%%%%%%%%%%%%%%%%%%%%%%%%%%%%%%%%%%%%%%%%%%%%%%%%%%%%%%%%%%%%%%%%%
\section{Introduction and main results}

Excited random walk (ERW) on $\IZ^d$ was introduced by Benjamini and
Wilson in \cite{BW03}. They proposed to modify the nearest neighbor
simple symmetric random walk by giving it a positive drift
(``excitation'') in the first coordinate direction upon reaching a
previously unvisited site. If the site had been visited before, then the 
walk made unbiased jumps to one of its nearest neighbor sites.  See
\cite{Ko03}, \cite{BR07}, \cite{HH08} and references therein for
further results about this particular model.
 
Zerner (\cite{Ze05}, \cite{Ze06}) generalized excited random walks by
allowing to modify the transition probabilities at each site not just
once but any number of times and, moreover, choosing them according to
some probability distribution. He obtained the criteria for recurrence
and transience and the law of large numbers for i.i.d.\ environments
on $\IZ^d$ and strips and also for general stationary ergodic
environments on $\IZ$. It turned out that this generalized model had
interesting behavior even for $d=1$, and this case was further studied
in \cite{MPV06}, \cite{BS08a}, \cite{BS08b}.

Results obtained in all these works rely on the assumption that
projections of all possible drifts on some fixed direction are
non-negative. In fact, the branching processes framework introduced in
\cite{KKS75} for random walks in random environment ($d=1$) and
employed in \cite{BS08a}, \cite{BS08b} for excited random walks, 
does not depend on the positivity
assumption, and it seems natural to use this approach for extending the
analysis to environments which allow both positive and negative
drifts. This was done in \cite{KZ08}, where the authors discussed
recurrence and transience, laws of large numbers, positive speed, and 
the averaged central limit theorem for multi-excited random walks on
$\IZ$ in i.i.d.\ environments with bounded number of ``excitations''
per site.  We postpone further discussion of known results for $d=1$
and turn to a precise description of the model considered in this
paper.

Given an arbitrary positive integer $M$ let 
\begin{align*}
  \Omega_M:=\big\{((\omega_z(i))_{i\in\IN})_{z\in\IZ}
 \mid \, \omega_z(i)\in[0,1],\ &\text{for }
   i\in\{1,2,\dots,M\}\\
  \text{\ and}\  
  \omega_z(i)=1/2,\  &\text{for } i>M,\  z\in\IZ\big\}.
\end{align*}
An element of $\Omega_M$ is called a cookie environment. For each
$z\in\IZ$, the sequence $\{\omega_z(i)\}_{i\in\IN}$ can be thought of
as a pile of cookies at site $z$, and $\omega_z(i)$ is referred to as
``the $i$-th cookie at $z$''. The number $\omega_z(i)$ is equal to the
transition probability from $z$ to $z+1$ of a nearest-neighbor random
walk upon the $i$-th visit to $z$. If $\omega_z(i)>1/2$ (resp.
$\omega_z(i)<1/2$) the corresponding cookie will be called positive
(resp.\ negative), $\omega_z(i)=1/2$ will correspond to a ``placebo''
cookie or, equivalently, the absence of an effective $i$-th cookie at
site $z$.

Let $\IP$ be a probability measure on $\Omega_M$, which satisfies the
following two conditions:
\begin{itemize}
 \item [(A1)] Independence: the sequence 
$(\omega_z(\cdot))_{z\in\IZ}$ is i.i.d.\ under $\IP$.
\item [(A2)] Non-degeneracy: 
\[ \IE\left[\prod_{i=1}^M
  \omega_0(i)\right]>0\ \text{ and }\ \IE\left[\prod_{i=1}^M
  (1-\omega_0(i))\right]>0.\]
\end{itemize}
Notice that we do not make any independence assumptions on the cookies
at the same site.

It will be convenient to define our ERW model using a
coin-toss construction.  Let $(\Sigma,\mathcal{F})$ be some
measurable space equipped with a family of probability measures
$P_{x,\omega},\ x\in\IZ,\ \omega\in\Omega_M$, such that for each
choice of $x\in\IZ$ and $\omega\in\Omega_M$ we have $\pm 1$-valued
random variables $B_{i}^{(z)},\ z\in\IZ,\ i\ge 1,$ which are
independent under $P_{x,\omega}$ with distribution given by
\begin{equation}\label{B}
 P_{x,\omega}(B_{i}^{(z)}=1)=\omega_z(i)\quad\mbox{ and }\quad
P_{x,\omega}(B_{i}^{(z)}=-1)=1-\omega_z(i).
\end{equation} 
Let $X_0$ be a random variable on $(\Sigma,\mathcal{F},P_{x,\omega})$ such that
$P_{x,\omega}(X_0=x)=1$.  Then an ERW starting at
$x\in\IZ$ in the environment $\omega$, $X:=\{X_n\}_{n\ge 0}$, can be defined on
the
probability space $(\Sigma,\mathcal{F},P_{x,\omega})$ by the relation 
\begin{eqnarray}
X_{n+1}&:=&X_n+B_{\#\{r\le n\mid X_r=X_n\}}^{(X_n)},\quad n\ge 0.\label{x1}
\end{eqnarray}
Informally speaking, upon each visit to a site the walker eats a
cookie and makes one step to the right or to the left with
probabilities prescribed by this cookie. Since $\omega_z(i)=1/2$ for
all $i>M$, the walker will make unbiased steps from $z$ starting from
the $(M+1)$-th visit to $z$. 

Events $\{B_{i}^{(z)}=1\}$, $i\in\IN$, $z\in\IZ$, will be referred to
as ``successes'' and events $\{B_{i}^{(z)}=-1\}$ will be called
``failures''.

The consumption of a cookie $\omega_z(i)$ induces a drift of size
$2\omega_z(i)-1$ with respect to $P_{x,\omega}$. Summing up over all
cookies at one site and taking the expectation with respect to $\IP$
gives the parameter
\begin{equation}
  \label{D}
  \delta\ :=\ \IE\Bigg[\sum_{i\ge 1}(2\omega_0(i)-1)\Bigg]\ 
=\ \IE\left[\sum_{i=1}^M(2\omega_0(i)-1)\right],
\end{equation}
which we call the \textit{average total drift per site}. It
plays a key role in the classification of the asymptotic behavior of
the walk.

We notice that there is an obvious symmetry between positive
and negative cookies: if the environment $(\omega_z)_{z \in \IZ}$ is
replaced by $(\omega^\prime_z)_{z \in \IZ} $ where $\omega^\prime_z
(i) \ = \ 1- \omega_z (i)$, for all $i\in\IN,\ z\in\IZ$, then
$X':=\{X_n^\prime\}_{n \ge 0}$, the ERW corresponding to the new
environment, satisfies
\begin{equation}
X'\overset{\mathcal{D}}{=} -X,
\end{equation}
where $\overset{\mathcal{D}}{=}$ denotes the equality in distribution.  
Thus, it is sufficient to consider, say, only
non-negative $\delta$ (this, of course, allows both negative and 
positive cookies), and we shall always assume this to be the
case. 

Define the \textit{averaged} measure $P_x$ by setting $P_x(\ \cdot\
)=\IE\left(P_{x,\omega}(\ \cdot\ )\right)$. Below we summarize known
results about this model.
\begin{theorem}[\cite{KZ08}] Assume (A1) and (A2). 
  \begin{itemize}
  \item [(i)] If $\delta \in[0,1]$ then $X$ is recurrent, i.e.\ for
    $\IP$-a.a.\ $\omega$ it returns $P_{0,\omega}$-a.s.\ infinitely
    many times to its starting point. If $\delta >1$ then $X$ is
    transient to the right, i.e.\ for $\IP$-a.a.\ 
    $\omega$, $X_n\to\infty$ as $n\to\infty$ $P_{0,\omega}$-a.s..
  \item [(ii)] There is a deterministic $v\in[0,1]$ such that $X$
    satisfies for $\IP$-a.a.\ $\omega$ the strong law of large numbers,
  \begin{equation}\label{speed}
    \lim_{n\to\infty}
\frac{X_n}{n}=v\quad P_{0,\omega}\text{-a.s..}
  \end{equation}
Moreover, $v=0$ for $\delta\in[0,2]$ and $v>0$ for $\delta >2$.
\item [(iii)] If $\delta>4$ then the sequence
  \[B_t^n:=\frac{X_{[tn]}-
  [tn] v}{\sqrt{n}},\quad \text{ $t\ge 0$}
\]
converges weakly under $P_0$ to a non-degenerate Brownian
motion with respect to the Skorohod topology on the space of
c\'adl\'ag functions.
  \end{itemize}
\end{theorem}
This theorem does not discuss the rate of growth of the ERW when it is
transient but has zero linear speed ($1<\delta\le 2$). It also
leaves open the question about fluctuations when $\delta\le 4$.

The rate of growth of the transient cookie walk with zero linear speed
was studied in \cite{BS08b} for the case of deterministic spatially
homogeneous non-negative cookie environments. For further discussion
we need some notation for the limiting stable distributions
that appear below. Given $\alpha\in (0,2]$ and $b>0$, denote by
$Z_{\alpha,b}$ a random variable (on some probability space) whose
characteristic function is determined by the relation
  \begin{equation}
    \label{sl}
    \log Ee^{iuZ_{\alpha,b}}=
    \begin{cases}
      -b|u|^\alpha\left(1-i\frac{u}{|u|}\tan 
\left(\frac{\pi\alpha}{2}\right)\right),
      &\text{if } \alpha\ne
      1;\\
      -b|u|\left(1+\frac{2i}{\pi}\frac{u}{|u|}\log|u|\right),&\text{if }
      \alpha=1.
    \end{cases}
  \end{equation}
  Observe that $Z_{2,b}$ is a centered normal random variable
  with variance $2b$. The weak convergence with respect to $P_0$ will
  be denoted by $\Rightarrow$.
\begin{theorem}[\cite{BS08b}] \label{BS} Let $\omega_z(i)=p_i\in[1/2,1)$,
  $i\in\IN$ for all $z\in\IZ$, where $p_i=1/2$ for $i>M$, 
and $\delta$ be as in (\ref{D}), that is
  $\delta=\sum_{i=1}^M(2p_i-1)$.
 \begin{itemize}
 \item [(i)] If $\delta\in(1,2)$ then there is a positive constant $b$
   such that as
   $n\to\infty$ \[\frac{X_n}{n^{\delta/2}}\Rightarrow(Z_{\delta/2,b})^{-\delta/2}.\]
 \item [(ii)] If $\delta=2$ then  $(X_n\log n)/n$ converges in
   probability to some constant $c>0$.
 \end{itemize}
 The above results also hold if $X_n$ is replaced by $\sup_{i\le n}
 X_i$ or $\inf_{i\ge n}X_i$.
\end{theorem}
The proof of Theorem~\ref{BS} used the non-negativity of cookies,
though this assumption does not seem to be essential for most parts of
the proof. It is certainly possible that the approach presented in
\cite{BS08b} could yield the same results without the non-negativity
assumption.

The functional central limit theorem for ERWs with $\delta\in[0,1)$ in
stationary ergodic non-negative cookie environments was obtained in
\cite{Do08}. The limiting process is shown to be Brownian motion
perturbed at extrema (see, for example, \cite{Da99}, \cite{CD99}).

The main results of this paper deal with the case when
$\delta\in(2,4]$, though they apply also to
$\delta\in(1,2]$. Moreover, our approach provides an alternative proof
of Theorem~\ref{BS} for general cookie environments that satisfy
conditions (A1) and (A2) (see Remark~\ref{alt}).

We establish the following theorem. 
\begin{theorem}
  \label{mainth}
  Let $T_n=\inf\{j\ge 0\,|\,X_j=n\}$ and $v$ be the speed of the ERW (see 
(\ref{speed})). The following statements hold under the averaged
measure $P_0$.
  \begin{itemize}
  \item [(i)] If $\delta\in(2,4)$ then there is a constant $b>0$ such
    that as $n\to\infty$
    \begin{align}
      &\frac{T_n-v^{-1}n}{n^{2/\delta}}\Rightarrow Z_{\delta/2,b},
\label{itime}\quad\text{and}
\\&\frac{X_n-vn}{n^{2/\delta}}\Rightarrow
-v^{1+2/\delta}Z_{\delta/2,b}.\label{ix}
    \end{align} 
  \item [(ii)] If $\delta=4$ then there is a constant $b>0$ such that
    as $n\to\infty$
    \begin{align}
      &\frac{T_n-v^{-1}n}{\sqrt{n\log n}}\Rightarrow Z_{2,b}
      ,\label{iitime}\quad\text{and}
      \\&\frac{X_n-vn}{\sqrt{n\log n}}\Rightarrow -v^{3/2}Z_{2,b}.\label{iix}
    \end{align}
  \end{itemize}
  Moreover, (\ref{ix}) and (\ref{iix}) hold if $X_n$ is replaced by
  $\sup_{i\le n} X_i$ or $\inf_{i\ge n}X_i$.
\end{theorem}

The paper is organized as follows.  In Section~\ref{bp} we recall the
branching processes framework and formulate two statements
(Theorem~\ref{X} and Theorem~\ref{progeny}), from which we later infer
Theorem~\ref{mainth}. Section~\ref{soft} explains the idea of the
proof of Theorem~\ref{progeny} and studies properties of the
approximating diffusion process.  In Section~\ref{4} we determine
sufficient conditions for the validity of
Theorem~\ref{progeny}. Section~\ref{5} contains the main technical
lemma (Lemma~\ref{main}). It is followed by three sections, where we
use the results of Section~\ref{5} to verify the sufficient conditions
of Section~\ref{4} and prove Theorem~\ref{X}.  The proof of
Theorem~\ref{mainth} is given in Section~\ref{1from2}.  The Appendix
contains proofs of several technical results.

%%%%%%%%%%%%%%%%%%%%%%%%%%%%%%%%%%%%%%%%%%%%%%%%%%%%%%%%%%%%%%%%

\section{Reduction to branching processes}\label{bp}

Suppose that the random walk $\{X_n\}_{n\ge 0}$ starts at $0$. Since
$\delta\ge 0$, Lemma~5 of [15] implies that $P_0(T_n<\infty)=1$ for
all $n\in\mathbb{N}$. At first, we recall the framework used in
\cite{BS08a}, \cite{BS08b}, and \cite{KZ08}.  The main ideas go back
at least to \cite{K73} and \cite{KKS75}.

For $n\in\IN$ and $k\le n$ define 
\[D_{n,k}=\sum_{j=0}^{T_n-1}\I_{\{X_j=k,\ X_{j+1}=k-1\}},\]
the number of jumps from $k$ to $k-1$ before time $T_n$.
Then
\begin{equation}\label{Tn}
T_n=n+2\sum_{k\le n}D_{n,k}=n+2\sum_{0\le k\le n}D_{n,k}+2\sum_{k<0}D_{n,k}.
\end{equation} 
The last sum is bounded above by the total time spent by $X_n$ below $0$.
When $\delta>1$, i.e.\ $X_n$ is transient to the right, the
    time spent below $0$ is $P_0$-a.s.\ finite,
and, therefore, for any $\alpha>0$
\begin{equation}\label{irrel}
 \lim_{n\to\infty}\frac{\sum_{k<0}D_{n,k}}{n^\alpha}=0
\quad\text{$P_0$-a.s..}
\end{equation} 
This will allow us to conclude that for transient ERWs fluctuations of
$T_n$ are determined by those of $\sum_{0\le k\le n}D_{n,k}$, once we
have shown that the latter are of order $n^{2/\delta}$.

We now consider the ``reversed'' process
$\left(D_{n,n},D_{n,n-1}\dots,D_{n,0}\right)$. Obviously, $D_{n,n}=0$
for every $n\in\IN$. Moreover, given
$D_{n,n},D_{n,n-1},\dots,D_{n,k+1}$, we can write 
\begin{align*}
 D_{n,k}=&\sum_{j=1}^{D_{n,k+1}+1}(\#\ \text{of jumps from $k$ to 
$k-1$ between the $(j-1)$-th}\\&\text{
   and $j$-th jump from $k$ to $k+1$ before time $T_n$}),\ k=0,1,\dots,n-1.
\end{align*}
Here we used the observation that the number of jumps from $k$ to $k+1$
before time $T_n$ is equal to $D_{n,k+1}+1$ for all $k \le n-1$. The
expression ``between the $0$-th and the $1$-st jump'' above should be
understood as ``prior to the $1$-st jump''.

Fix an $\omega\in\Omega_M$ and denote by $F^{(k)}_m$ the number of
``failures'' in the sequence $B^{(k)}$ (see (\ref{B}) with $z$
replaced by $k$) before the $m$-th ``success''. Then, given
$D_{n,k+1}$,
\[D_{n,k}=F^{(k)}_{D_{n,k+1}+1}.\]

Since the sequences $B^{(k)}$, $k\in\IZ$, are i.i.d.\ under $P_0$,
we have that $F^{(k)}_m\overset{\mathcal{D}}{=}F^{(n-k-1)}_m$ and can
conclude that the distribution of
$\left(D_{n,n},D_{n,n-1}\dots,D_{n,0}\right)$ coincides with
that of $(V_0,V_1,\dots,V_n)$, where $V=\{V_k\}_{k\ge 0}$
is a Markov chain defined by
\[V_0=0,\quad V_{k+1}=F^{(k)}_{V_k+1},\quad k\ge 0.\] 

For $x\ge 0$ we shall denote by $[x]$ the integer part of $x$ and by
$P_x^V$ the measure associated to the process $V$, which starts with
$[x]$ individuals in the $0$-th generation.  Observe that $V$ is a
branching process with the following properties:

(i) $V$ has exactly $1$ immigrant in each generation (the 
immigration occurs before the reproduction) and, therefore, 
does not get absorbed at $0$.

(ii) The number of offspring of the $m$-th individual in
  generation $k$ is given by the number of failures
  between the $(m-1)$-th and $m$-th success in the sequence
  $B^{(k)}$. In particular, if $V_{k}\ge M$ then the offspring
  distribution of each individual after the $M$-th one is
 $\mathrm{Geom(1/2)}$ (i.e.\ geometric on $\{0\}\cup\IN$ with parameter
$1/2$). 

Therefore (here and throughout taking any sum from $k$ to $\ell$ for 
$k>\ell$ to be zero) we can write
\begin{equation}
 \label{defV}
V_{k+1}=\sum_{m=1}^{M\wedge (V_k+1)}
\zeta^{(k)}_m+\sum_{m=1}^{V_k-M+1}\xi^{(k)}_m,\quad k\ge 0,
\end{equation} 
where $\{\xi^{(k)}_m;\,k\ge 0,\,m\ge 1\}$ are i.i.d.\ $\mathrm{Geom(1/2)}$
random variables, vectors $\left(\zeta^{(k)}_1, \zeta^{(k)}_2, \cdots
  \zeta^{(k)}_M\right)$, $k\ge 0$, are i.i.d.\ under $P^V_x$ and
independent of $\{\xi^{(k)}_m;\,k\ge 0,\,m\ge 1\}$.  For each $k\ge 0$
the random variables $\{\zeta^{(k)}_m\}_{m=1}^M$ are neither
independent nor identically distributed, but, given that for some
$j<M$ \[\sum_{m=1}^j\zeta^{(k)}_m\ge M,\] that is all cookies at site $k$ have
been eaten before the $j$-th jump from $k$ to $(k+1)$, we are left with
$\{\zeta^{(k)}_m\}_{m=j+1}^M$ that are independent Geom(1/2) random variables.
Define
\begin{equation}
\label{ext}
\sigma^V_0=\inf\{j>0\,|\, V_j=0\},
\quad S^V=\sum_{j=0}^{\sigma^V_0-1}V_j.     
\end{equation} 

Detailed information about the tails of $\sigma^V_0$ and $S^V$ will
enable us to use the renewal structure and characterize the behavior
of $\sum_{0\le k\le n}D_{n,k}$, and, therefore, of $T_n$ as
$n\to\infty$ for transient ERWs. We shall show in Section~\ref{1from2}
that the following two statements imply Theorem~\ref{mainth}.

\begin{theorem}\label{X}
Let $\delta>0$. Then
\begin{equation}
\label{ub}
\lim _{n \rightarrow \infty} n^\delta P_0^V(\sigma^V_0>n) = C_1\in(0,\infty).
\end{equation}
\end{theorem}

\begin{theorem}
  \label{progeny}
Let $\delta>0$. Then 
\begin{equation}
  \label{ptails}
  \lim_{n\to\infty}n^{\delta/2}P_0^V\left(S^V>n\right)=C_2 \in (0, \infty).
\end{equation}
\end{theorem}
\begin{remark}{\em In fact, a weaker result than (\ref{ub}) is
    sufficient for our purpose: there is a constant $B$ such that $n
    ^\delta P_0^V (\sigma^V_0>n)\leq B $ for all $n\in\IN$ (see
    condition (A) in Lemma~\ref{4.1}).  We also would like to point
    out that the limits in (\ref{ub}) and (\ref{ptails}) exist for
    every starting point $x\in\IN\cup \{0\}$ with $C_1$ and $C_2$
    depending on $x$. The proofs simply repeat those for $x=0$.

For the model described in Theorem~\ref{BS}, the convergence
(\ref{ub}) starting from $x\in\IN$ is shown in \cite[Proposition
3.1]{BS08b} (for $\delta\in(1,2)$), and (\ref{ptails}) for
$\delta\in(1,2]$ is the content of \cite[Proposition
4.1]{BS08b}. Theorem~\ref{X} can also be derived from the construction
in \cite{KZ08} (see Lemma 17) and \cite{FYK90}.  We
use a different approach and obtain both results directly without
using the Laplace transform and Tauberian theorems.  }
\end{remark}

We close this section by introducing some additional
notation.  For $x\ge 0$ we set
\begin{align}
 \tau^V_x&=\inf\{j> 0\,|\, V_j\ge x\};\label{tau}\\
\sigma^V_x&=\inf\{j> 0\,|\,V_j\le x\}.\label{sig}
\end{align}
We shall drop the superscript whenever there is no possibility of
confusion.

When the random walk $X$ is transient to the right,
$P^V_y(\sigma^V_0<\infty)=1$ for every $y\ge 0$. This implies that
$P^V_y(\sigma^V_x<\infty)=1$ for every $x\in[0,y)$.

Let us remark that when we later deal with a continuous process on
$[0,\infty)$ we shall simply use the first hitting time of $x$ to
record the entrance time in $[x,\infty)$ (or $[0,x]$), given that the
process starts outside of the mentioned interval. We hope that
denoting the hitting time of $x$ for such processes also by $\tau_x$
will not result in ambiguity.

%%%%%%%%%%%%%%%%%%%%%%%%%%%%%%%%%%%%%%%%%%%%%%%%%%%%%%%%%%%%%%%%%%%%%%%%

\section{The approximating diffusion process and its
  properties}\label{soft}

The bottom-line of our approach is that the main features of 
branching process $V$ killed upon reaching $0$ are reasonably well described by
a simple diffusion process. 

The parameters of such diffusion processes can be easily computed at the
heuristic level. For $V_k\ge M$, (\ref{defV}) implies that 
\begin{equation}\label{diffV}
V_{k+1}-V_k=\sum_{m=1}^M \zeta^{(k)}_m-M+1+\sum_{m=1}^{V_k-M+1}(\xi_m^{(k)}-1).
\end{equation}
By conditioning on the number of successes in the first $M$ tosses it is easy to
compute (see Lemma 3.3 in \cite{BS08a} or Lemma 17 in \cite{KZ08} for details)
that for all $x\ge 0$ 
\begin{equation}\label{drift}
E^V_x\left(\sum_{m=1}^M \zeta^{(k)}_m-M+1\right)=1-\delta.
\end{equation}
The term $\sum_{m=1}^M \zeta^{(k)}_m-M+1$ is independent of 
$\sum_{m=1}^{V_k-M+1}(\xi_m^{(k)}-1)$. When $V_k$ is large, the latter is
approximately normal with mean $0$ and variance essentially equal to $2V_k$. 

Therefore, the relevant diffusion should be given by the following
stochastic differential equation:
\begin{equation}\label{sde}
dY_t=(1-\delta)\,dt+\sqrt{2 Y_t}\,dB_t,\quad Y_0=y>0, \quad t\in[0,\tau_0^Y],
\end{equation}
where for $x\ge 0$ we set
\begin{equation}
  \label{hit}
  \tau_x^Y=\inf\{t\ge 0\,|\,Y_t=x\}.
\end{equation}
Throughout the rest of the paper, unless stated otherwise, we shall
assume that $\delta>0$. Observe that $\tau_0^Y<\infty$ a.s., since
$2Y_t$ is a squared Bessel process of of dimension  $2(1-\delta)<2$
(for a proof, set $a=0$ and let $b\to\infty$ in part (ii) of
Lemma~\ref{elem} below).

The above heuristics are justified by the next lemma. 
\begin{lemma} \label{weakcon}
Let $Y=\{Y_t\}_{t\ge 0}$ be the solution of (\ref{sde}).
Fix an arbitrary $\epsilon>0$. For $y\in(\epsilon,\infty )$ let $V_0=[ny]$, 
and define 
\begin{equation*}
Y^{\epsilon,n}_t = \frac{V_{[nt]\wedge \sigma^V_{\epsilon n}}}{n},\quad
t\in[0,\infty),
\end{equation*}
where $\sigma^V_x$ is given by (\ref{sig}).
Then the sequence of processes $Y^{\epsilon,n}=\{Y^{\epsilon,n}_t\}_{t\ge 0}$
converges in distribution as $n\to\infty$ with respect to the Skorokhod topology
on the space of c\`adl\`ag functions to the stopped diffusion
$Y^\epsilon=\{Y_{t\wedge\tau^Y_\epsilon}\}_{t\ge 0}$, $Y_0=y$.
\end{lemma}

\begin{proof}
  We simply apply the (much more general) results of \cite{EK86}.  We
  first note that our convergence result considers the processes up to
  the first entry into $(- \infty ,\epsilon ]$ for $\epsilon > 0$
  fixed.  So we can choose to modify the rules of evolution for $V$
  when $V_k \leq \epsilon n $: we consider the process $(V^{n,\epsilon}_k)_{k
    \geq 0}$ where, with the existing notation,
\begin{equation}
  V^{n,\epsilon}_0=[ny],\quad V^{n,\epsilon}_{k+1}=\sum_{m=1}^{M} 
  \zeta^{(k)}_m+\sum_{m=1}^{V^{n,\epsilon}_k \vee
    (\epsilon n)-M+1}\xi^{(k)}_m,\quad k\ge 0.
\end{equation} 
Then (given the regularity of points for the limit process) it will suffice to
show the convergence of processes
\begin{equation*}
\tilde{Y}^{\epsilon,n}_t = \frac{V^{n,\epsilon}_{[nt]}}{n},\quad t\in[0,\infty),
\end{equation*}
to the solution of the stochastic integral equation
\begin{equation}
dY_t=(1-\delta)\,dt+\sqrt{2 (Y_t \vee \epsilon)}\,dB_t,\quad Y_0=y>0, \quad
t\in[0, \infty ).
\end{equation}
We can now apply Theorem 4.1 of Chapter 7 of \cite{EK86} with $X_n(t) =
\tilde{Y}^{\epsilon,n}_t $.  The needed
uniqueness of the martingale problem corresponding to operator
\begin{equation}
Gf \ = \ (x \vee \epsilon ) f^{\prime \prime} + (1- \delta ) f^\prime
\end{equation}
follows from \cite{EK86}, Chapter 5, Section 3 (Theorems 3.6 and 3.7 imply the
distributional uniqueness for solutions of the corresponding stochastic integral
equation, and Proposition 3.1 shows that this implies the uniqueness for
the martingale problem).
\end{proof}

We shall see in a moment that this diffusion has the desired behavior
of the extinction time and of the total area under the path before the
extinction (see Lemma~\ref{difext} and Lemma~\ref{C} below).
Unfortunately, these properties in conjunction with
Lemma~\ref{weakcon} do not automatically imply Theorem~\ref{X} and
Theorem~\ref{progeny}, and work needs to be done to ``transfer'' these
results to the corresponding quantities of the process $V$.
Nevertheless, Lemma~\ref{weakcon} is very helpful when $V$ stays large
as we shall see later.

\medskip

In the rest of this section we state and prove several facts about
$Y$. When we need to specify that the process $Y$ starts at $y$ at
time $0$ we shall write $Y^y$.  Again, whenever there is no ambiguity
about which process is being considered we shall drop the superscript
in $\tau^Y_x$ defined in (\ref{hit}).

\begin{lemma}
\label{elem}
Fix $y>0$.
\begin{itemize}
\item [(i)] (Scaling) Let $\widetilde{Y}=\{\widetilde{Y}_t\}_{t\ge 0}$, where
$\widetilde{Y}_t=\dfrac{Y^y_{ty}}{y}$.
Then $\widetilde{Y}\overset{\mathcal{D}}{=}Y^1$.
\item [(ii)] (Hitting probabilities) Let $0\le a<y<b$. Then
  \[P_y^{Y}(\tau_a<\tau_b)=\frac{b^\delta-y^\delta}{b^\delta-a^\delta}.\]
\end{itemize}
\end{lemma}

\begin{proof}
Part (i) can be easily checked by It\^o's formula applied to $\widetilde{Y}_t$
or seen from scaling properties of the generator. The proof of part (ii) is
standard once we notice that the process $(Y^y_t)^{\delta}$ stopped upon
reaching the boundary of $[a,b]$ is a martingale. We omit the details.
\end{proof}

\begin{lemma}\label{difext}
Let $Y$ be the diffusion process defined by (\ref{sde}). Then
\[\lim_{x\to\infty}x^\delta P_1^Y(\tau_0>x)=C_3\in (0,\infty).\]
\end{lemma}
\begin{proof}
For every $\epsilon>0$ and for all $x>1/\epsilon$ we have by
Lemma~\ref{elem}
\begin{align*}
  x^\delta P^Y_1(\tau_0>x)&\ge x^\delta P^Y_1(\tau_0>x\,|\,\tau_{\epsilon
    x}<\tau_0) P^Y_1(\tau_{\epsilon x}<\tau_0) \\& \ge x^\delta
  P^Y_{\epsilon x}(\tau_0>x)\left(\epsilon x\right)^{-\delta}=
  \epsilon^{-\delta} P^Y_1(\tau_0>\epsilon^{-1})>0.
\end{align*}
This implies that for each $\epsilon>0$
\[\liminf_{x\to\infty}x^\delta P^Y_1(\tau_0>x)\ge \epsilon^{-\delta}
P^Y_1(\tau_0>\epsilon^{-1})>0.
\]
Taking the $\limsup_{\epsilon\to 0}$ in the right-hand side we get
\[\liminf_{x\to\infty}x^\delta P^Y_1(\tau_0>x)\ge \limsup_{\epsilon\to
  0} \epsilon^{-\delta} P^Y_1(\tau_0>\epsilon^{-1})=\limsup_{x\to
  \infty}x^\delta P^Y_1(\tau_0>x).\] This would immediately imply the
existence of a finite non-zero limit if we could show that
 \[\limsup_{x\to \infty}x^\delta
 P^Y_1(\tau_0>x)<\infty.\]
This is the content of the next lemma. 

\begin{lemma}
\label{difup}
Let $Y$ be the diffusion process defined by
      (3.3). Then \[\limsup_{x\to \infty}x^\delta
      P^Y_1(\tau_0>x)<\infty.\]
\end{lemma}
The proof is very similar to the proof of the discrete version (see
(A) in Lemma~\ref{4.1} and its proof in Section~\ref{px}) and, thus,
is omitted.
\end{proof}

The final result of this section can be viewed as the ``continuous 
counterpart'' of Theorem~\ref{progeny}. 
It concerns the area under the path of $Y$.
\begin{lemma} \label{C} Let $Y$ be the diffusion process defined by (\ref{sde}).
Then
\[\lim_{y\to\infty}y^\delta P_1^Y\left(\int_0^{\tau_0}Y_t\,dt>y^2\right)
=C_4\in(0,\infty).\]
\end{lemma}
\begin{proof}
  The proof uses scaling and follows the same steps as the proof of
  Lemma~\ref{difext}. For every $\epsilon>0$ and $y>1/\epsilon$ we
  have
\begin{align*}
  y^\delta P_1^Y&\left(\int_0^{\tau_0}Y_t\,dt>y^2\right)\ge y^\delta
  P_1^Y\left(\int_0^{\tau_0}Y_t\,dt>y^2\,\Big|\,\tau_{\epsilon
      y}<\tau_0\right) P_1^Y(\tau_{\epsilon y}<\tau_0)\\ \ge\,&y^\delta
  P_{\epsilon y}^Y\left(\int_0^{\tau_0}Y_t\,dt>y^2\right)(\epsilon
  y)^{-\delta} =\epsilon^{-\delta}P_{\epsilon
    y}^Y\left(\int_0^{\tau_0/(\epsilon y)} Y_{\epsilon ys}\,
    ds>\frac{y}{\epsilon}\right) \\=\,&\epsilon^{-\delta}P_{\epsilon
    y}^Y\left(\int_0^{\tau_0/(\epsilon y)} \frac{Y_{\epsilon
        ys}}{\epsilon y}\, ds>\epsilon^{-2}\right)
  =\epsilon^{-\delta} P_1^Y\left(\int_0^{\tau_0}
    Y_s\,ds>\epsilon^{-2}\right)>0.
\end{align*}
This calculation, in fact, just shows that \[y^\delta
P_1^Y\left(\int_0^{\tau_0}Y_t\,dt>y^2\right)\] is a non-decreasing
positive function of $y$. Therefore, we only need to prove that it is
bounded as $y\to\infty$. But for $y>1$
\begin{align*}
&y^\delta P_1^Y\left(\int_0^{\tau_0}Y_t\,dt>y^2\right)\\&=
 P_1^Y\left(\int_0^{\tau_0}Y_t\,dt>y^2\,\Big|\,
 \tau_y<\tau_0\right)+
 y^\delta P_1^Y\left(\int_0^{\tau_0}Y_t\,dt>y^2,\,\tau_y>\tau_0\right)
 \\[2mm]&\le 1+y^\delta P_1^Y\left(\tau_0>y,\,
 \tau_y>\tau_0\right)\le 1+y^\delta P_1^Y\left(\tau_0>y\right).
\end{align*}
An application of Lemma~\ref{difup} finishes the proof.
\end{proof}

%%%%%%%%%%%%%%%%%%%%%%%%%%%%%%%%%%%%%%%%%%%%%%%%%%%%%%%%%%%%%%%%%%%%%%%%%

\section{Conditions which imply Theorem~\ref{progeny}}\label{4}
 
We have shown that the diffusion process $Y$ has the desired asymptotic behavior
of the area under the path up to the exit time $\tau^Y_0$.
In this section we give sufficient conditions under which we can
``transfer'' this result to the process $V$ and obtain Theorem~\ref{progeny}.
\begin{lemma} \label{4.1}
  Suppose that
\begin{itemize}
\item[(A)] There is a constant $B$ such that $ n ^\delta
P^V_0(\sigma _0 > n)\le B$ for all $n\in\IN$;
\item[(B)] For every $\epsilon>0$ \[\lim_{n\to\infty} P_{\epsilon n}^V
  \left( \sum_{i=0} ^{\sigma_0-1} V_i>n^2\right) = P_1^Y\left(
    \int_0^{\tau_0} Y_t\,dt>\epsilon^{-2}\right);\]
\item[(C)] $\displaystyle\lim_{n\to\infty}n^\delta P_0^V\left(\tau_n<\sigma_0
\right)=C_5$.
\end{itemize}
Then \[\lim_{n\to\infty}n^\delta P_0^V\left(\sum_{i=0} ^{\sigma_0-1}
  V_i>n^2\right)=C_4C_5,\] where $C_4$ is the constant from Lemma~\ref{C}.
\end{lemma}
\begin{proof}
  Fix an $\epsilon\in(0,1)$ and split the path-space of $V$ into two parts,
  the event $H_{n,\epsilon}:=\left\{\tau_{\epsilon n
    }^V<\sigma_0^V\right\}$ and its complement, $H_{n,\epsilon}^c$.

  First, consider the behavior of the total progeny on the event
  $H_{n,\epsilon}^c$.  On $H_{n,\epsilon}^c$, the process $V$ stays
  below $\epsilon n$ until the time $\sigma_0$.  Estimating each
  $V_i$ from above by $\epsilon n$ and using (A) we get for all 
$n\in\IN$
  \[n^\delta P_0^V \left(\sum_{i=0} ^{\sigma_0-1} V_i>n^2,\
    H_{n,\epsilon}^c\right) \le n^\delta P_0^V(\sigma_0>n/\epsilon)\le
  (2\epsilon)^\delta B. \]  
Therefore, for all $n\in\IN$
\begin{align*}
  0\le
  n^\delta P_0^V \left(\sum_{i=0} ^{\sigma_0-1} V_i>n^2\right)-n^\delta P_0^V
\left(\sum_{i=0} ^{\sigma_0-1} V_i>n^2,\
    H_{n,\epsilon}\right)\le(2\epsilon)^\delta B.
\end{align*}
Hence, we only need to deal with the total progeny on the event
$H_{n,\epsilon}$.  The rough idea is that, on $H_{n,\epsilon}$, it
is not unnatural for the total progeny to be of order $n^2$. This means that the
decay of the probability that the total progeny is over $n^2$ 
comes from the decay of the
probability of $H_{n,\epsilon}$, which is essentially given by
condition (C). This would suffice if we could let $\epsilon=1$ but
we need $\epsilon$ to be small, thus, some scaling is necessary to
proceed with the argument, and this brings into play condition 
(B) and the result of Lemma~\ref{C}.

To get a lower bound on $F_n:=n^\delta P_0^V \left(\sum_{i=0} ^{\sigma_0-1}
  V_i>n^2,\ H_{n,\epsilon}\right)$ we use monotonicity of $V$
with respect to the initial number of particles, conditions
(B) and  (C), and Lemma~\ref{C}:
\begin{multline*}
  \lim_{\epsilon\to 0}\liminf_{n\to\infty} F_n
  \ge \lim_{\epsilon\to 0}\lim_{n\to\infty}n^\delta
    P_0^V(H_{n,\epsilon})P_{\epsilon n}^V
  \left(\sum_{i=0} ^{\sigma_0-1} V_i>n^2\right)\\
  =C_5\lim_{\epsilon\to 0}\epsilon^{-\delta}P_1^Y\left( \int_0^{\tau_0}
    Y_t\,dt>\epsilon^{-2}\right)=C_4C_5.
\end{multline*}

For an upper bound on $F_n$ we shall need two more parameters,
$K\in(1,1/\epsilon)$ and $R>1$. At the end, after taking the limits as
$n\to\infty$ and then $\epsilon\to 0$ we shall let $K\to\infty$ and
$R\to 1$.
\begin{multline*}
   n^{-\delta} F_n=
  P_0^V \left(\sum_{i=0}^{\tau_{ \epsilon n }-1}V_i+
    \sum_{i=\tau_{\epsilon n }}
    ^{\sigma_0-1}V_i>n^2,\ H_{n,\epsilon}\right)\\
  \le P_0^V \left(\sum_{i=\tau_{ \epsilon n }}
    ^{\sigma_0-1}V_i>n^2(1-K\epsilon),\ \sum_{i=0}^{\tau_{\epsilon n
      }-1}V_i\le K\epsilon n^2,\ H_{n,\epsilon}\right)\\+P_0^V
 \left(\sum_{i=0}^{\tau_{ \epsilon n
      }-1}V_i>K\epsilon n^2,\ H_{n,\epsilon}\right)
\end{multline*}
We bound the first term on the right-hand side by the following sum:
\begin{multline*}
 P_0^V
  \left(\sum_{i=\tau_{\epsilon n }} ^{\sigma_0-1}V_i>n^2(1-K\epsilon),\
    V_{\tau_{\epsilon n }}\le R\epsilon n,\
    H_{n,\epsilon}\right)\\+P_0^V
  \left(\sum_{i=\tau_{\epsilon n }} ^{\sigma_0-1}V_i>n^2(1-K\epsilon),\
    V_{\tau_{\epsilon n }}>R\epsilon n,\
    H_{n,\epsilon}\right).
\end{multline*}
Estimating these terms in an obvious way and putting everything back
together we get
\begin{align*}
n^{-\delta} F_n&\le P_{R\epsilon n}^V \left(\sum_{i=0}
    ^{\sigma_0-1}V_i>n^2(1-K\epsilon)\right)P_0^V(H_{n,\epsilon})
\\&+P_0^V \left(V_{\tau_{\epsilon n }}>R\epsilon n,\
    H_{n,\epsilon}\right)+P_0^V \left(\sum_{i=0}^{\tau_{\epsilon n
      }-1}V_i>K\epsilon n^2,\ H_{n,\epsilon}\right)\\&=(I)+(II)+(III).
\end{align*}
It only remains to multiply everything by $n^\delta$ and consider the
upper limits.

Term $n^\delta(I)$ gives the upper bound $C_4C_5$ in the same way as
we got a lower bound by sending $n\to\infty$, $\epsilon\to 0$, and
then $R\to 1$ and using easily verified continuity properties of the
relevant distributions. Parameter $K$ disappears when we let $\epsilon\to 0$.

Term $(II)$ is exponentially small in $n$ for fixed $\epsilon$ and $R$
(see Lemma~\ref{overshoot} below), thus $n^\delta (II)$ goes to zero as
$n\to\infty$.

Finally, since $V_i\le \epsilon n$ for all $i<\tau_{\epsilon n}$, we get
\begin{align*}
  n^\delta P_0^V \left(\sum_{i=0}^{\tau_{\epsilon n }-1}V_i>K\epsilon n^2,\
    H_{n,\epsilon}\right)&\le n^\delta P_0^V\left(\tau_{\epsilon n
    }>Kn, H_{n,\epsilon}\right)\\ \le n^\delta
  P_0^V\left(\sigma_0>Kn, H_{n,\epsilon}\right) &\le n^\delta
  P_0^V(\sigma_0>Kn)\le \frac{2^\delta B}{K^\delta}. \qedhere
\end{align*}
\end{proof}

%%%%%%%%%%%%%%%%%%%%%%%%%%%%%%%%%%%%%%%%%%%%%%%%%%%%%%%%%%%%%%%%%%%%%%%

\section{Main tools}\label{5}

The main result of this section is Lemma~\ref{main}, which is a
discrete analog of Lemma~\ref{elem} (ii). 

We start with two technical lemmas. The first one will be used many times
throughout the paper.
\begin{lemma}
\label{overshoot}
There are constants $c_1,c_2>0$ and $N\in\mathbb{N}$ such that for
every $x\ge N$ and $y\ge 0$, 
\begin{align}
  \sup_{0\le z<x}P_z^V\left(V_{\tau_x}>x+y\,|\,\tau_x<\sigma_0\right)&\le 
c_1( e^{-c_2y^2/x} + e^{-c_2y});\label{over} \\ \sup_{x<z< 4x} 
P_z^V(V_{\sigma_x\wedge \tau_{4x}}<x-y)&\le c_1
  e^{-c_2y^2/x}\label{under}.
\end{align}
\end{lemma}
\noindent This statement is a consequence of the fact that the offspring
distribution of $V$ is essentially geometric. The proof is
given in the Appendix.

\begin{lemma}\label{scale}
Fix $a \in (1,2]$.  Consider the process $V$ with $|V_0
-a^n| \le a^{2n/3}$ and let $\gamma=\inf
\{k\ge 0\,|\, V_k \notin (a^{n-1},a^{n+1} ) \}$.  Then for all
sufficiently large $n$
\begin{align*}
  (&i)\quad P^V\Big(\mathrm{dist}(V_\gamma ,(a^{n-1},a^{n+1})) \ge
  a^{2(n-1)/3}\Big) \le \exp(-a^{n/4});\\(i&i)\quad
  \left|P^V(V_\gamma\leq a^{n-1}) - \frac{a^\delta}{a^\delta +1}\right| \leq
a^{-n/4}.
\end{align*}
\end{lemma}
\noindent Part (i) is an immediate consequence of Lemma~\ref{overshoot}.
The proof of part (ii) is basic but technical and is given in the Appendix.
\begin{lemma}[Main lemma]\label{main}
  For each $a \in (1,2]$ there is an $\ell_0\in\IN$ such that if
  $\ell,m,u,x\in\mathbb{N}$ satisfy $\ell_0\le\ell <m< u$ and $|x-a^m|\le
  a^{2m/3}$ then
  \[\frac{h_a^-(m)-1}{h_a^-(u)-1}\le
  P_x^V(\sigma_{a^\ell}>\tau_{a^u})\le
  \frac{h_a^+(m)-1}{h_a^+(u)-1},\] where
  \[h_a^{\pm}(i)=\prod_{r=\ell+1}^i\left(a^\delta\mp a^{-\lambda
      r}\right),\quad i>\ell,\] and $\lambda$ is some small positive number not
  depending on $\ell$.
\end{lemma}

\begin{remark}\label{rml}
 {\em It is to be noted that for fixed $\ell$ there are $K_1(\ell)$
   and $K_2(\ell)$ such that  
\[K_1(\ell) \le  \frac{h^{\pm}_a(i)}{a^{(i-\ell)\delta}} \leq
K_2(\ell)\quad\text{for all } i >\ell\] and $K_j(\ell)\to 1$ as
   $\ell \rightarrow \infty $, $j=1,2$.}
\end{remark}

\begin{proof}[Proof of Lemma \ref{main}] We will show the upper bound
  by comparing the process $V$ with another process $\widetilde{V}$,
  whose exit probabilities can be estimated by further reduction to an
  exit problem for a birth-and-death-like Markov chain.

  For $i\in\IN$ set $x_i=[a^i+a^{2i/3}]$. By monotonicity, it is
  enough to prove the upper bound when the starting point $x$ is equal
  to $x_m$. Thus, we set $V_0=x_m$.  The comparison will be done in two
  steps.

{\em Step 1.} We shall construct a sequence of stopping times $\gamma_i$, $i\ge
0$, and a comparison process $\widetilde{V}=(\widetilde{V}_k)_{k\ge 0}$ with
$x_\ell$ as an absorbing point so that $\widetilde{V}_k\ge V_k$ for all $k$
before the absorption.  Let $\gamma_0=0$,
\[\gamma_1=\inf\{k>0\,|\,V_k\not\in(a^{m-1},a^{m+1})\},\quad
\widetilde{V}_k=V_k\ \text{ for }\ k=0,1,\dots,\gamma_1-1,\] and at time
$\gamma_1$ add to $V_{\gamma_1}$ the necessary number of particles to get
\begin{equation*}
 \widetilde{V}_{\gamma_1}=
 \begin{cases}
x_{m-1},&\text{if } V_{\gamma_1}\le a^{m-1};\\
x_{m+j},&\text{if }
x_{m+j-1}<V_{\gamma_1}\le x_{m+j},\ j\in\IN.
  \end{cases}
\end{equation*}
Clearly, $\widetilde{V}_{\gamma_1}\ge V_{\gamma_1}$. By construction,
$\widetilde{V}_{\gamma_1}=x_n$ for some $n\ge m-1,\ n\ne m$.  If
$\widetilde{V}_{\gamma_1}=x_\ell$, then we stop the
process. 

Assume that we have already defined stopping times $\gamma_r$,
$r=0,1,\dots,i$, and the process $\widetilde{V}_k$ for all $k\le
\gamma_i$ so that $\widetilde{V}_{\gamma_i}=x_n$ for some $n>
\ell$. We define $\widetilde{V}_k$ for $k>\gamma_i$ by applying to it
the same branching mechanism as for $V$, namely, (\ref{defV}) with $V$
replaced by $\widetilde{V}$, $k\ge \gamma_i$.  Denote by
$\gamma_{i+1}$ the first time after $\gamma_i$ when $\widetilde{V}$
exits the interval $(a^{n-1},a^{n+1})$. At time $\gamma_{i+1}$, if the
process exited through the lower end of the interval then we set
$\widetilde{V}_{\gamma_{i+1}}=x_{n-1}$, if the process exited the
through the upper end we add to $\widetilde{V}$ the minimal number of
particles needed to get $\widetilde{V}_{\gamma_{i+1}}=x_s$ for some
$s>n$. If $\widetilde{V}_{\gamma_{i+1}}=x_\ell$, then we stop the
process.  Thus, we obtain a sequence of stopping times $\gamma_i$,
$i\ge 0$, and the desired dominating process $\widetilde{V}$ absorbed
at $x_\ell$ such that $\widetilde{V}_{\gamma_i}\in
\{x_\ell,x_{\ell+1},\dots\}$, $i\ge 0$.

{\em Step 2.} Define a Markov chain $R=(R_j)_{j\ge 0}$ on
$\{\ell,\ell+1,\dots\}$ by setting \[R_j=n\quad \text{if} \quad
\widetilde{V}_{\gamma_j}=x_n,\quad j\ge 0.\] The state $\ell$ is
absorbing. Let $\sigma^R_\ell=\inf\{j\ge 0\,|\, R_j=\ell\}$ and
$\tau^R_u=\inf\{j\ge 0\,|\, R_j\ge u\}$. By
construction, \[P^V_{x_m}(\sigma^V_{a^\ell}>\tau^V_{a^u})\le
P^{\widetilde{V}}_{x_m}(\sigma^{\widetilde{V}}_{x_\ell}>\tau^{\widetilde{V}}_{x_u})=
P^R_m(\sigma^R_\ell>\tau^R_u).\] We shall show that
$(h^+_a(R_j))_{j\ge 0}$ is a supermartingale with respect to the
natural filtration. (We set $h^+_a(\ell)=1$.)  The optional stopping
theorem and monotonicity of function $h^+_a$ will immediately imply
the upper bound in the statement of the lemma.

For $i >\ell$ we have
\begin{multline*}
  E^R_i\left(h_{a}^{+} (R_{1})\right) = h_{a}^{+} (i-1) P^R_i (R_1 =
  i-1) \\+ h_{a}^{+} (i+1) P^R_i (R_1 = i+1) +
  \sum_{n =i+2}^\infty h_{a}^{+} (n) P^R_i (R_1=n) 
\end{multline*}
By the definition of $h^+_a$ this is less or equal than
\begin{multline}\label{h+}
  h_{a}^{+}(i) \bigg[(a^{\delta} - a^{-\lambda i})^{-1}P^R_i (R_1 = i-1) 
\\+(a^{\delta} - a^{-\lambda(i+1)}) P^R_i (R_1 =
    i+1) + \sum_{n = i+2}^\infty
    a^{\delta(n-i)} P^R_i (R_1 = n)\bigg].
\end{multline}
By Lemma~\ref{scale} and Lemma~\ref{overshoot} we have that for
all $i>\ell$, where $\ell$ is chosen sufficiently large,
\begin{align*}
  &P^R_i (R_1 = i-1)=\frac{a^\delta}{a^\delta+1}+O(a^{-i/4}),\\
  &P^R_i (R_1 = i+1)= \frac{1}{a^\delta+1}+O(a^{-i/4}),
  \ \text{and }\\
  &P^R_i (R_1 \ge n)\le P^R_{n-2}(R_1 \ge n)=O(\exp(-a^{n/4}))\
  \text{for all }n\ge i+2.
\end{align*}
Substituting this into (\ref{h+}) and performing elementary
computations we obtain 
\begin{equation*}
  E^R_i\left(h_{a}^{+} (R_{1})\right) \le h_{a}^{+}(i)
  \bigg[1-\frac{a^{-\lambda
i}}{a^\delta+1}\,\left(a^{-\lambda}-a^{-\delta}\right)
   + O\big(a^{-2\lambda i}\big)\bigg]
\le h^+_a(i),
\end{equation*}
provided that $\lambda\le\min\{1/8,\delta\}$ and $\ell$ (therefore
$i$) is sufficiently large.

For the lower bound we argue in a similar manner, except that now we
choose $x_m=[a^m-a^{2m/3}]+1$, assume that $V_0=x_m$, and construct a
comparison process $\widetilde{V}=(\widetilde{V}_k)_{k\ge 0}$ absorbed
at $x_\ell$ so that $\widetilde{V}_0=V_0$ and $\widetilde{V}_k\le V_k$
for all $k$ before the absorption.

More precisely, we let $\gamma_0=0$ and assume that we have already
defined stopping times $\gamma_r$, $r=0,1,\dots,i$, and the process
$\widetilde{V}_k$ for all $k\le \gamma_i$ so that
$\widetilde{V}_{\gamma_i}=x_n$ for some $n\ne\ell$. We define
$\widetilde{V}_k$ for $k>\gamma_i$ by (\ref{defV}) with $V$ replaced
by $\widetilde{V}$, $k\ge \gamma_i$.  Denote by $\gamma_{i+1}$ the
first time after $\gamma_i$ when $\widetilde{V}$ exits the interval
$(a^{n-1},a^{n+1})$. At time $\gamma_{i+1}$, if the process exited
through the upper end of the interval we set
$\widetilde{V}_{\gamma_{i+1}}=x_{n+1}$, if the process exited
through the lower end we reduce the number of particles by
removing the minimal number of particles to ensure that
$\widetilde{V}_{\gamma_{i+1}}=x_s$ for some $s<n$. If
$\widetilde{V}_{\gamma_{i+1}}\le x_\ell$, then we stop the process and redefine
$\widetilde{V}_{\gamma_{i+1}}$ to be $x_\ell$.  This
procedure allows us to obtain a sequence of stopping times $\gamma_i$,
$i\ge 0$, and the desired comparison process $\widetilde{V}$ 
absorbed at $x_\ell$ such that $V_{\gamma_i}\in \{x_\ell,x_{\ell+1},\dots\}$,
$i\ge 0$. 

Next, just as in the proof of the upper bound, we construct a Markov
chain $R=(R_j)_{j\ge 0}$ and show that $(h^-_a(R_j))_{j\ge 0}$ is a
submartingale (with $h^-_a(\ell)$ defined to be $1$).  The optional
stopping theorem and monotonicity of function $h^-_a$ imply the lower
bound.
\end{proof}

\begin{corollary} \label{boundtau} For each non-negative integer $x$ there
  exists a constant $C_6 = C_6(x)$ such that for every $n\in\IN$
\begin{equation}\label{bt1}
 n^\delta P^V_x( \tau_n < \sigma _0) \leq C_6.
\end{equation} 
Moreover, for each $\epsilon>0$ there is a constant $c_3=c_3(\epsilon)$ such
that for all $n\in\IN$
\begin{equation}\label{bt2}
   P^V_n(\sigma_0 > \tau_{c_3n}) < \epsilon.
\end{equation}
\end{corollary}
\begin{remark}
  {\em In fact, (\ref{bt1}) will be substantially improved by
    Lemma \ref{triv}.}
\end{remark}
\begin{proof}
  We choose arbitrarily $a \in (1,2]$ and an $\ell\ge \ell_0$ as in
  Lemma \ref{main} but also such that $a^\ell > x $.  We note that it is
  sufficient to prove the statement for
  $n$ of the form $[a^u]$.  We define stopping times $\beta_i$,
  $i\in\IN$,  by
  \begin{align*}
    \beta _1&= \inf \{k> 0\,|\, V_k\ge a^{\ell+1}\}; \\
    \beta _{i+1}& = \inf \{k > \beta_i\,: \,V_k \ge a^{\ell+1} \mbox{
      and } \exists
  s\in(\beta_i,k)\,|\, V_s \le a^\ell\}.
  \end{align*}
Lemma \ref{overshoot} and the monotonicity of $V$ with respect to its
starting point imply that 
\begin{multline*}
  P^V_x\left(\exists r \in [\beta_i, \beta_{i+1})\,|\, V_r \ge a^u
    \,|\,\beta_i < \sigma_0\right) \\
  \le \frac{h_a^+(\ell+1)-1}{h_a^+(u)-1} + \sum_{k=\ell+1} ^\infty
    P^V_x (V_{\beta_i} \ge a^k + a^{2k/3}\,|\,\beta_i < \sigma_0)\,
    \frac{h_a^+(k+1)-1}{h_a^+(u)-1}  \\=
  \frac{h_a^+(\ell+1)-1}{h_a^+(u)-1} \left(1 + \sum_{k=\ell+1} ^\infty
    P^V_x (V_{\beta_i} \ge a^k + a^{2k/3}\,|\,\beta_i < \sigma_0)\,
    \frac{h_a^+(k+1)-1}{h_a^+(\ell+1)-1} \right)\\\le
  \frac{2(h_a^+(\ell+1)-1)}{h_a^+(u)-1}
\end{multline*}
supposing, as we may, that $\ell$ was fixed sufficiently large.  Thus, 
\[
a^{u \delta }P^V_x(\sigma_0 > \tau_{a^u}) \le 2a^{u \delta }\,
\frac{(h_a^+(\ell+1)-1 )}{h_a^+(u)-1}\sum_{i=1}^\infty P^V_x(\beta_i <
\sigma_0) .
\]
The bound (\ref{bt1}) now follows from noting that $P^V_x(\beta_i < \sigma_0)$
decays
geometrically fast to zero (with a rate which may depend on $\ell$ but
does not depend on $u$) and that $a^{u \delta }(h_a^+(u)-1)^{-1}$ is
bounded in $u$ (see Remark~\ref{rml}).

To prove (\ref{bt2}) we notice that by Lemma~\ref{main} and (\ref{bt1}) for all
$n>a^\ell$
\begin{align*}
P^V_n(\sigma_0 > \tau_{c_3n})&\le P^V_n(\sigma_0 >
\tau_{c_3n},\,\sigma_{a^\ell}>\tau_{c_3 n})+ P^V_n(\sigma_0 >
\tau_{c_3n},\,\sigma_{a^\ell}<\tau_{c_3 n})\\&\le
P^V_n(\sigma_{a^\ell}>\tau_{c_3
n})+\frac{C_6(a^\ell)}{(c_3n)^{\delta}}=O(c_3^{-\delta}).
\end{align*}
The constant $c_3$ can be chosen large enough to get (\ref{bt2}) for all
$n\in\IN$. 
\end{proof}

%%%%%%%%%%%%%%%%%%%%%%%%%%%%%%%%%%%%%%%%%%%%%%%%%%%%%%%%%%%%%%%%%%%%%%%

\section{Proof of (A)}\label{px}

\begin{proposition}\label{reentry}
 There is a constant $c_4>0$ such that for all $k,\,x\in\IN$ and $y\ge 0$ 
\begin{equation}
  \label{ind}
 P^V_y \left( \sum_{r=1}^{\sigma_0}\I_{\{V_r \in [x,2x)\}} >
  2xk\right) \le P^V_y(\rho_0<\sigma_0)(1- c_4)^k. 
\end{equation}
where $\rho_0=\inf\{j\ge 0\,|\, V_j\in[x,2x)\}$.
\end{proposition}
\begin{proof}
First, observe that there is a constant $c>0$ such that for all
$x\in\IN$
\begin{equation*}
  (i) \quad P^V_{2x}(
  \sigma_{x/2} <  x) > c; 
\qquad(ii)\quad P^V_{x/2}( \sigma_0 < \tau_x) > c.
\end{equation*}
The inequality $(i)$ is an immediate consequence of
Lemma~\ref{weakcon}. To prove the second inequality, we fix
$x_0\in\IN$ and let $x>2x_0+1$. Then by Corollary~\ref{boundtau}
\begin{align*}
  P^V_{x/2}( \sigma_0 < \tau_x)&= P^V_{x/2}( \sigma_0 < \tau_x\,|\,
  \sigma_{x_0}<\tau_x)P^V_{x/2}( \sigma_{x_0} < \tau_x)\\&\ge
  (1-C_6(x_0)x^{-\delta}) P^V_{x/2}( \sigma_{x_0} < \tau_x).
\end{align*}
Choosing $x_0$ large enough
and applying Lemma~\ref{main} to the last term in the right-hand side
we obtain $(ii)$ for all sufficiently large $x$. Adjusting the constant
$c$ if necessary we can extend $(ii)$ to all $x\in\IN$.

Next, we show that $(i)$ and $(ii)$ imply (\ref{ind}) with $c_4=c^2$. 
Denote by $\rho_0\ge 0$ the first entrance time of $V$ in $[x,2x)$ and
set \[\rho_j=\inf\{r\ge \rho_{j-1}+2x\,|\, V_r\in [x,2x)\},\quad
j\ge 1.\] Notice that for each $j\ge 1$, the time spent by $V$
in $[x,2x)$ during the time interval $[\rho_{j-1},\rho_j)$ 
is at most $2x$.  If $V$ spends more than $2xk$ units of time in
$[x,2x)$ before time $\sigma_0$ then $\rho_k<\sigma^V_0$. Thus,
\begin{multline}\label{st}
  P^V_y\left( \sum_{r=1}^{\sigma_0}\I_{\{V_r \in [x,2x)\}} >
  2xk\right) \le P^V_y(\rho_k<\sigma_0)
    \\= P^V_y(\rho_k<\sigma_0\,|
  \,\,\rho_{k-1}<\sigma_0)P^V_y(\rho_{k-1}<\sigma_0).
\end{multline}
Using the strong Markov property, monotonicity with respect to the
starting point, and inequalities $(i)$ and $(ii)$ we get
\begin{align*}
  P^V_y&(\rho_k<\sigma_0\,| \,\,\rho_{k-1}<\sigma_0)\le \max_{x\le
    z<2x} P^V_z(\rho_1<\sigma_0)\\&\le \max_{x\le z<2x}
  \left(P^V_z(\rho_1<\sigma_0,\,
    \sigma_{x/2}<x)+P^V_z(\rho_1<\sigma_0,\, \sigma_{x/2}\ge
 x)\right)\\&\le \max_{x\le
    z<2x}\left(P^V_z(\rho_1<\sigma_0\,|\,
    \sigma_{x/2}<x)P^V_z(\sigma_{x/2}<x)+1-P^V_z(\sigma_{x/2}<
 x)\right) \\&\le \max_{x\le z<2x}
  \left(1-P^V_z(\sigma_{x/2}<x)(1-P^V_z(\rho_1<\sigma_0\,|\,
    \sigma_{x/2}<x))\right)
 \\&\le 1-P^V_{2x}(\sigma_{x/2}<x)(1-P^V_{x/2}(\rho_0<\sigma_0))
 \\&\le 1-P^V_{2x}(\sigma_{x/2}<x)P^V_{x/2}(\tau_x>\sigma_0)
    \le 1-c^2.
\end{align*}
Substituting this in (\ref{st}) and iterating in $k$ gives (\ref{ind}).
\end{proof}

\begin{proposition} \label{NEW} For every $h>0$
\[
\lim_{\epsilon\to 0}\limsup_{n \to \infty} n ^\delta P^V_0 \left( \sum
_{i=1}^{\sigma_0}
\I_{\{V_i < \epsilon n\}} > n h\right)=0.
\]
\end{proposition}

\begin{proof}
Fix $\epsilon\in(0,1)$ and let $k=k(n,\epsilon)$ be
  the smallest integer such that $2^k \ge \epsilon n$. Define
  intervals $I_i = [2^{k-i}, 2^{k-i+1})$, $i\in\IN$, and events
\[A_i=\left\{
\sum_{r=1}^{\sigma_0} \I_{\{V_r \in I_i\}} > \frac{h2^{i-1}}
{\epsilon i(i+1)}\,|I_i|\right\}.
\] 
Intervals $I_i$ and events $A_i$ depend on $n$ but this is not reflected
in our notation.  Since $2^{k-1}<\epsilon n$ and
$\sum_{i=1}^k(i(i+1))^{-1}< 1$, we have
\[\left\{ \sum _{i=1}^{\sigma_0} \I_{\{V_i < \epsilon n\}} > n h\right\}\subset
\bigcup_{i=1}^kA_i,\] and, therefore,
\[
P_0^V\left(\sum _{i=1}^{\sigma_0} \I_{\{V_i < \epsilon n\}} > n h\right)  \le  
\sum_{i= 1}^k P(A_i).
\]
Using (\ref{bt1}) and (\ref{ind}) we get
\begin{align*}
  n^\delta P_0^V\left(\sum _{i=1}^{\sigma_0} \I_{\{V_i < \epsilon n\}}
 > n h\right)&\le 
  n^\delta\sum_{i=1}^k \left(1-c_4\right) 
^{\left[\frac{h2^{i-2}}{\epsilon i(i+1)}
    \right]}P^V_0(\tau_{2^{k-i}}<\sigma_0)
\\&\le  C_6(0)\,\epsilon^{-\delta} 
  \sum_{i\ge 1} \left(1-c_4\right) ^{\left[\frac{h2^{i-2}}{\epsilon i(i+1)} 
    \right]}2^{i \delta},
\end{align*}
and this quantity vanishes as $\epsilon\to 0 $.
\end{proof}

\begin{proof}[Proof of (A)]
To obtain (A) we apply (\ref{bt1}) and Proposition~\ref{NEW} to the right-hand
side of the following inequality:
\begin{align*}
n^\delta P^V_0 (\sigma_0 >n) &\leq n^\delta P^V_0 (\tau_{\epsilon n} < \sigma
_0) +n^\delta P^V_0 (\sigma_0 >n, \tau_{\epsilon n} > \sigma _0 )\\&\le
\epsilon^{-\delta}(\epsilon n)^{\delta}P^V_0 (\tau_{\epsilon n} < \sigma
_0)+n^{\delta}P^V_0\left( \sum _{i=1}^{\sigma_0}
\I_{\{V_i < \epsilon n\}} > n\right). \qedhere
\end{align*}
\end{proof}

%%%%%%%%%%%%%%%%%%%%%%%%%%%%%%%%%%%%%%%%%%%%%%%%%%%%%%%%%%%%%%%%%%%%%%%%%

\section{Proof of  (B)}

We shall need the following fact.
\begin{proposition}\label{T}
 For each $\epsilon> 0$, there is a constant $C_7=C_7(\epsilon)>0$ such 
that 
  \[
  P^V_n(\sigma_0 > C_7n) < \epsilon\quad\text{for all $n\in\IN$}.
  \]
\end{proposition}

\begin{proof}
We have
\[P^V_n(\sigma_0 > C_7n)\le P^V_n(\sigma_0 >
\tau_{c_3n})+P^V_n(\sigma_0 > C_7n,\,\sigma_0 <\tau_{c_3n}).\] Using
(\ref{bt2}) we can choose $c_3>1$ so that $P^V_n (\sigma_0 > \tau_{c_3n}
) <\epsilon / 2$ for all $n\in\IN$.  Thus, we only need to estimate
the last term. Notice that it is bounded above by the probability that
the occupation time of the interval $(0,c_3n)$ up to the moment
$\sigma_0$ exceeds $C_7n$. The latter can be estimated by Markov
inequality:
\[P^V_n\left( \sum_{r=1}^{\sigma_0} \I_{\{V_r < c_3n\}}>C_7n
\right)\le (C_7n)^{-1}E^V_n\left( \sum_{r=1}^{\sigma_0} \I_{\{V_r <
    c_3n\}} \right).\] We claim that the last expectation does not
exceed $4nc_3/c_4$ and so we can take $C_7>8c_3/(\epsilon c_4)$.
Indeed, let $m$ be the smallest positive integer such that $2^m\ge
c_3n$. Then writing the expectation of our non-negative integer-valued
random variable as the sum of the probabilities of its tails and using
(\ref{ind}) to estimate the tails we get
\begin{equation*}
 E^V_n\left( \sum_{r=1}^{\sigma^V_0} \I_{\{V_r < c_3n\}}  \right) \le
 \sum_{j=1}^mE^V_n\left( \sum_{r=1}^{\sigma^V_0}
 \I_{\{V_r\in[2^{j-1},2^j)\}}  \right)
\le \sum_{j=1}^m\frac{2^j}{c_4}
\le \frac{2^{m+1}}{c_4}\le\frac{4nc_3}{c_4}. \qedhere
\end{equation*}
\end{proof}

\begin{proof}[Proof of  (B)] For every $\alpha\in(0,\epsilon)$ and
  $\beta\in(0,1)$ we have
\begin{multline}\label{Bee}
  P^V_{\epsilon n}\left(\sum_{j=0}^{\sigma_{\alpha n}-1}
    V_j>n^2\right)\le
  P^V_{\epsilon n}\left(\sum_{j=0}^{\sigma_0-1} V_j>n^2\right)\\
  \le P^V_{\epsilon n}\left(\sum_{j=0}^{\sigma_{\alpha n}-1}
    V_j>(1-\beta)n^2\right)+P^V_{\epsilon
    n}\left(\sum_{j=\sigma_{\alpha n}}^{\sigma_0-1} V_j>\beta
    n^2\right).
\end{multline}
By Lemma~\ref{weakcon} for every $R>0$ 
\begin{equation}\label{Be}
  \lim_{n\to\infty} P^V_{\epsilon
 n}\left(\sum_{j=0}^{\sigma_{\alpha n}-1} V_j>Rn^2\right)=
 P^Y_\epsilon\left (\int_0^{\sigma_\alpha} 
Y_s\,ds>R\right),
\end{equation}
since, as is easily verified, under law $ P^Y_\epsilon $ the law of $
\int_0^{\sigma_\alpha} Y_s\,ds $ has no atoms.  Next, we notice that
for all $x,\,\beta>0$
\begin{multline}\label{Bef}
 P^Y_\epsilon\left
 (\int_0^{\sigma_0} Y_s\,ds>(1+\beta)x\right) \ - \ P^Y_\epsilon\left
 (\int_{\sigma_\alpha}^{\sigma_0} Y_s\,ds>\beta x\right)\\\le
P^Y_\epsilon\left
 (\int_0^{\sigma_\alpha} Y_s\,ds>x\right)\le 
P^Y_\epsilon\left
 (\int_0^{\sigma_0} Y_s\,ds>x\right).
\end{multline}
By the strong Markov property and scaling, 
\begin{align}\label{Beg}\nonumber
 P^Y_\epsilon\left
 (\int_{\sigma_\alpha}^{\sigma_0} Y_s\,ds> \beta x\right)&=P^Y_\alpha
\left
 (\int_0^{\sigma_0} Y_s\,ds>\beta x\right)\\&=
P^Y_1\left
 (\int_0^{\sigma_0} Y_s\,ds> \beta x\alpha^{-2}\right)\to
 0\ \ \text{as } \alpha\to 0.
\end{align}
Letting $n\to\infty$, then $\alpha\to 0$, and finally $\beta\to 0$ we
obtain from (\ref{Bee})--(\ref{Beg})
that \[\liminf_{n\to\infty}P^V_{\epsilon
  n}\left(\sum_{j=0}^{\sigma_0-1} V_j>n^2\right)\ge P^Y_1\left
  (\int_0^{\sigma_0} Y_s\,ds>\epsilon^{-2}\right).\] 
To get the matching upper bound it is enough to show that for every $\nu>0$ and
$\beta\in(0,1)$ there is an $\alpha\in(0,\epsilon)$ such that for all
sufficiently large $n$
\begin{equation}
\label{Beh}
 P^V_{\epsilon n}
\left(\sum_{j=\sigma_{\alpha n}}^{\sigma_0-1} V_j>\beta n^2\right)
<2\nu.
\end{equation}
The left-hand side of (\ref{Beh}) does not exceed
\[P^V_{\alpha n} (\tau_{c_3n\alpha}<\sigma_0)+P^V_{\alpha n}
\left(\sum_{j=0}^{\sigma_0-1} V_j>\beta
  n^2,\,\tau_{c_3n\alpha}>\sigma_0\right).\] Given $\nu$, define
$c_3(\nu)$ and $C_7(\nu)$ as in (\ref{bt2}) and Proposition~\ref{T}
respectively. Let $\alpha<\sqrt{\beta/(c_3 C_7)}$.  By (\ref{bt2}) the
first term above is less than $\nu$ for all $n>1/\alpha$. On the set
$\{\tau_{c_3n\alpha}>\sigma_0\}$ the process $V$ is below $c_3n\alpha$
and, thus, the second term is bounded above by $P^V_{\alpha
  n}(\sigma_0>(\beta n)/(c_3 \alpha))$. By Proposition~\ref{T} and our
choice of $\alpha$ the latter probability does not exceed $\nu$.
Using relations (\ref{Bee})--(\ref{Beg}) and, again, the absence of
atoms for the distribution of $\int_0^{\sigma_x} Y_s\,ds$ under
$P^Y_1$ for each $x\ge 0$, we get the desired upper bound.
\end{proof}

%%%%%%%%%%%%%%%%%%%%%%%%%%%%%%%%%%%%%%%%%%%%%%%%%%%%%%%%%%%%%%%%%%%%%%%%%

\section{Proofs of  (C) and Theorem \ref{X}}
First we prove (C) of Lemma~\ref{4.1}. Then using the approach of
Lemma~\ref{4.1} we show the convergence claimed in Theorem \ref{X}.

The next lemma includes (C) as a special case ($k=0$, $C_5=f(0)$). 
\begin{lemma}\label{triv}
There is a function $f:\,\mathbb{N}\cup\{0\}\to(0,\infty)$,
    such that \[\lim_{n\to\infty} n^\delta
    P_k^V(\sigma_0>\tau_n)=f(k)\quad \text{for each integer } k\ge 0.\]  
\end{lemma}
We shall need the following proposition.
\begin{proposition}
\label{s}
For each $a\in(1,2]$ and $k\ge 0$
  \begin{equation}
    \label{ser}
    \sum_{j=1}^\infty \Big|a^\delta
  P_k^V(\sigma_0>\tau_{a^j}\,|\, \sigma_0>\tau_{a^{j-1}})-1\Big|<\infty.
  \end{equation}
\end{proposition}

\begin{proof}
  By the monotonicity in the initial number of particles we get a
  lower bound: for all sufficiently large $j$ \[P_k^V(\sigma_0>\tau_{a^j}\,|\,
  \sigma_0>\tau_{a^{j-1}})\ge P_{a^{j-1}}^V(\sigma_0>\tau_{a^j}).\] For
  an upper bound we need to take into account the possibility of a
  large ``overshoot''. Let $x=a^{j-1}+a^{\frac{2(j-1)}{3}}$, then
  \begin{align*}
    P_k^V&(\sigma_0>\tau_{a^j}\,|\, \sigma_0>\tau_{a^{j-1}})\le
    P_x^V(\sigma_0>\tau_{a^j})\\
    &+P_k^V(\sigma_0>\tau_{a^j}\,|\, \sigma_0>\tau_{a^{j-1}},
    V_{\tau_{a^{j-1}}}>x)P_k^V(V_{\tau_{a^{j-1}}}>x\,|\,
    \sigma_0>\tau_{a^{j-1}})\\&\le
    P_x^V(\sigma_0>\tau_{a^j})+P_k^V(V_{\tau_{a^{j-1}}}>x\,|\,
    \sigma_0>\tau_{a^{j-1}}).
  \end{align*}
  The last probability decays faster than any power of $a^{-j}$ as
  $j\to\infty$ by Lemma~\ref{overshoot}. Therefore, it is enough to
  show the convergence of the series
  \begin{equation*}
    \sum_{j=1}^\infty \Big|a^\delta
  P_{x_{j-1}}^V(\sigma_0>\tau_{a^j})-1\Big|,
  \end{equation*}
where $x_j=a^j+\epsilon_j$
  and $0\le \epsilon_j\le a^{\frac{2j}{3}}$.  We have for all sufficiently large
$j$ (with $\ell$ chosen appropriately for $a$ as in Lemma \ref{main}),
  \begin{align*}
    \Big| P_{x_{j-1}}^V&(\sigma_0>\tau_{a^j})-a^{-\delta}\Big|\\&\le\Big|
    P_{x_{j-1}}^V(\sigma_{a^\ell}>\tau_{a^j})-a^{-\delta}\Big|+
P_{x_{j-1}}^V(\sigma_0>\tau_{a^j}>\sigma_{a^\ell})\\&\le\Big|
    P_{x_{j-1}}^V(\sigma_{a^\ell}>\tau_{a^j})-a^{-\delta}\Big|+
P_{x_{j-1}}^V(\sigma_0>\tau_{a^j}\,|\,\tau_{a^j}>\sigma_{a^\ell})\\&\le\Big|
    P_{x_{j-1}}^V(\sigma_{a^\ell}>\tau_{a^j})-a^{-\delta}\Big|+
    P_{a^\ell}^V(\sigma_0>\tau_{a^j}).
  \end{align*}
  By Lemma \ref{main} \[\sum_{j=1}^\infty\Big|
  P_{x_{j-1}}^V(\sigma_{a^\ell}>\tau_{a^j})-a^{-\delta}\Big|<\infty,\]
  and to complete the proof of (\ref{ser}) we invoke the bound 
provided by (\ref{bt1}) for $x = [a^\ell]$.
\end{proof}

\begin{proof}[Proof of Lemma~\ref{triv}]
  Fix an arbitrary non-negative integer $k$ and $a\in(1,2]$.  For each
  $n>a$ there is an $m\in\mathbb{N}$ such that $a^m\le n<a^{m+1}$. We
  have
  \[a^{m\delta}P_k^V(\sigma_0>\tau_{a^{m+1}})\le n^\delta
  P_k^V(\sigma_0>\tau_n)\le a^{(m+1)\delta}P_k^V(\sigma_0>\tau_{a^m}).\]
  If we can show that
  \begin{equation}
\label{gak}
\lim_{m\to\infty}
  a^{m\delta}P_k^V(\sigma_0>\tau_{a^m})=g(a,k)>0,
\end{equation}
for some $g(a,k)$, then 
\begin{align*}
 0<a^{-\delta}g(a,k)&\le
\liminf_{n\to\infty}n^\delta P_k^V(\sigma_0>\tau_n)\\&\le
\limsup_{n\to\infty} n^\delta P_k^V(\sigma_0>\tau_n)\le a^\delta
g(a,k).
\end{align*}
This implies
\[1\le\frac{\limsup_{n\to\infty} n^\delta P_k^V(\sigma_0>\tau_n)}
{\liminf_{n\to\infty}n^\delta P_k^V(\sigma_0>\tau_n)}\le a^{2\delta},\]
and we obtain the claimed result by letting $a$ go to $1$. 

To show (\ref{gak}) we set $\ell=\min\{j\in\IN\,|\, a^j>k\}$ and notice that for
$m>\ell$
  \begin{align*}
    a^{m\delta}P_k^V(\sigma_0>\tau_{a^m})&=a^\delta
    P_k^V(\sigma_0>\tau_{a^m}\,|\,\sigma_0>\tau_{a^{m-1}})\times
    a^{(m-1)\delta}
    P_k^V(\sigma_0>\tau_{a^{m-1}})\\=\dots&=a^{\ell
 \delta}P_k^V(\sigma_0>\tau_{a^\ell})\times
    \prod_{j=\ell+1}^m a^\delta
    P_k^V(\sigma_0>\tau_{a^j}\,|\,\sigma_0>\tau_{a^{j-1}}).
  \end{align*}
  Since all terms in the last product are strictly positive and $a^{\ell
 \delta}P_k^V(\sigma_0>\tau_{a^\ell})$ does not depend on $m$, the
  convergence (\ref{gak}) follows from (\ref{ser}).
\end{proof}

\begin{proof}[Proof of Theorem~\ref{X}]
  We will show that $\lim\limits_{n \rightarrow \infty } n^\delta
  P^V_0 (\sigma _0 > n) = C_3C_5$, where $C_3$ and $C_5$ are the same as in
  Lemma~\ref{difext} and condition (C). 

  We begin with a lower bound. Fix positive $\epsilon$ and $\beta \ll
  \epsilon$. We have
\[
P^V _0(\sigma _0 > n) \geq  P^V _0(\tau _{n
  \epsilon} < \sigma _ 0) P^V_{\epsilon n}(\sigma_0 > n )\ge 
P^V _0(\tau _{\epsilon n} < \sigma _ 0) P^V_{\epsilon n}(\sigma_{\beta
  n} > n ).
\]
By (C), Lemma~\ref{weakcon}, and scaling (Lemma \ref{elem} (i)) 
\[
\liminf_{n \rightarrow \infty } n ^\delta P^V _0(\sigma _0 > n)\ge 
C_5\,\epsilon^{-\delta} P^Y_{1}\left(\tau^Y_{\beta/\epsilon}
 >\epsilon^{-1}\right).
\]
Letting $\beta\to 0$ and then $\epsilon\to 0$ we obtain via
Lemma~\ref{difext} that \[\liminf_{n \rightarrow
  \infty } n ^\delta P^V _0(\sigma _0 > n) \geq C_3 C_5 .\]

The upper bound is slightly more complicated. First, notice that
\[n^\delta P^V_0(\sigma_0>n,\,\tau_{\epsilon n} >\sigma_0)\le n^\delta
P^V_0\left(\sum_{i=1}^{\sigma_0}\I_{\{V_i <\epsilon n\}}>n\right).\]
By Proposition~\ref{NEW} the right-hand side becomes negligible as
$n\to\infty$ and then $\epsilon \to 0$. Thus, it is enough to estimate $n^\delta
P^V_0(\sigma_0>n,\,\tau_{\epsilon n} <\sigma_0)$.
Let $R\in(1,3/2)$. Then
\begin{align*}
n ^\delta P^V _0(\sigma _0 > n,\tau_{\epsilon n}<\sigma_0)&\leq 
n ^\delta P^V _0(\sigma_0>\tau_{\epsilon n} >(R-1)n)\\ &+  
n ^\delta P^V _0(V_{\tau_{\epsilon n }} > R \epsilon n,\tau_{\epsilon
n}<\sigma_0 )\\&+ 
n ^\delta P^V _0( \sigma_0 - \tau_{\epsilon n} > (2-R)n,
V_{\tau_{\epsilon n }}\leq R \epsilon n ).
\end{align*}
By Proposition~\ref{NEW} the first term on the right-hand side
vanishes for every fixed $R>1$ when we let $n\to \infty$ and then
$\epsilon\to 0$.  By Lemma~\ref{overshoot} the $\limsup_{n\to\infty}$
of the second term is zero.  Thus it will be sufficient to bound the
last term.  For $\beta\ll\epsilon$ let us define $\sigma^{\epsilon
  n}_{\beta n} $ to be the the first time after $\tau_{\epsilon n} $
that $V$ falls below $\beta n$.  Then the last term is bounded above
by
\begin{multline*}
 n ^\delta \left(P^V _0( \sigma_0 - \sigma^{\epsilon n}_{\beta n}  
> (R-1)n)+ 
P^V _0(\sigma^{\epsilon n}_{\beta n}  - \tau_{\epsilon n} > (3-2R)n,
V_{\tau_{\epsilon n }}\leq R \epsilon n ) \right) \\\le \left(P^V _{\beta n}(
\sigma_0> (R-1)n)+  P^V _{R\epsilon n}(\sigma_{\beta n} > (3-2R)n)\right)n
^\delta P^V_0(\tau_{\epsilon n}<\sigma_0).
\end{multline*}
Taking $\limsup_{n\to\infty}$ and then letting $\beta\to 0$ we obtain
(by Proposition~\ref{T}, Lemma~\ref{weakcon}, and (C)) the following
upper bound for $\limsup_{n \rightarrow \infty }n ^\delta P^V _0(
\sigma_0 - \tau_{\epsilon n} > (2-R)n, V_{\tau_{\epsilon n }}\leq R
\epsilon n )$,
\[C_5\epsilon^{-\delta}P_{R\epsilon}^Y(\tau_0>(3-2R))=
C_5\epsilon^{-\delta}P_1^Y(\tau_0>\epsilon^{-1}(3-2R)/R).\] As $\epsilon\to 0$
and then $R\to 1$, the latter expression converges by Lemma~\ref{difext} to
$C_5C_3$. This completes the proof.
\end{proof}
%%%%%%%%%%%%%%%%%%%%%%%%%%%%%%%%%%%%%%%%%%%%%%%%%%%%%%%%%%%%%%%%%%%%%%%%

\section{Proof of Theorem~\ref{mainth}}\label{1from2}

Let $\delta>2$. By (\ref{Tn}) and (\ref{irrel})), it is enough to show
that as $n\to\infty$
\begin{equation}
 \label{ed}\frac{2\sum_{k=0}^n D_{n,k}-(v^{-1}-1)n}{n^{2/\delta}}\
 \overset{\mathcal{D}}{=}\ \frac{2\sum_{j=0}^n
 V_j-(v^{-1}-1)n}{n^{2/\delta}}
\end{equation}
converges in distribution to $Z_{\delta/2,b}$ for some $b>0$. Define
the consecutive times when $V_j=0$,\[\sigma_{0,0}=0,\quad
\sigma_{0,i}=\inf\{j>\sigma_{0,i-1}\,|\,V_j=0\},\quad i\in\IN,\] the
total progeny of $V$ over each lifetime,
$S_i=\sum_{j=\sigma_{0,i-1}}^{\sigma_{0,i}-1}V_j,\ i\in\IN$, and the
number of renewals up to time $n$, $N_n=\max\{i\ge
0\,|\,\sigma_{0,i}\le n\}$. Then
$(\sigma_{0,i}-\sigma_{0,i-1},S_i)_{i\ge 1}$ are i.i.d.\ under
$P^V_0$. Moreover,
$\sigma_{0,i}-\sigma_{0,i-1}\overset{\mathcal{D}}{=}\sigma^V_0$ and
$S_i\overset{\mathcal{D}}{=}S^V$, $i\in\IN$. By Theorem~\ref{progeny} the
distribution of $S^V$ is in the domain of attraction of the law of
$Z_{\delta/2,\tilde{b}}$ for some $\tilde{b}>0$ (see, for example \cite{Dur96},
Chapter 2, Theorem~7.7.) Since by Theorem~\ref{X} (in fact, the upper
bound (A) is sufficient) the second moment of $\sigma^V_0$ is finite,
it follows from standard renewal theory (see, for example,
\cite{Gu88}, Theorems~II.5.1 and II.5.2) that
\[\frac{N_n}{n}\overset{\mathrm{a.s.}}{\to}\lambda:=
\left(E^V_0\sigma_0\right)^{-1},\] and for each $\epsilon>0$ there is
$c_5>0$ such that $P^V_0\left(|N_n-\lambda
  n|>c_5\sqrt{n}\right)<\epsilon$ for all sufficiently large
$n$. Using the fact that $T_n/n\to v^{-1}$ a.s.\ as $n\to\infty$ and
relations (\ref{Tn}) and (\ref{ed}) we get that
$E^V_0S_i=(v^{-1}-1)/(2\lambda)$.

\textit{Proof of part (i).} Let $\delta\in(2,4)$. We have
\begin{equation*}
 \frac{\sum_{j=0}^nV_j-(v^{-1}-1)n/2}{n^{2/\delta}}
=\frac{\sum_{i=1}^{N_n}(S_i-E^V_0S_i)}{n^{2/\delta}}+
E^V_0S_1\,\frac{N_n-\lambda n}{n^{2/\delta}}+\frac{\sum_{j=\sigma_{0,N_n}}^n
V_j}{n^{2/\delta}}.
\end{equation*}
By Theorem I.3.2, \cite{Gu88}, the first term converges in
distribution to $Z_{\delta/2,\tilde{b}}$. The second term converges to
zero in probability by the above mentioned facts from renewal
theory. The last term is bounded above by $S_{N_{n+1}}/n^{2/\delta}$,
which converges to zero in probability.  This finishes the proof of
(\ref{itime}), which immediately gives (\ref{ix}) with $X_n$ replaced
by $\sup_{i\le n}X_i$, since $\{\sup_{i\le n} X_i< m\}=\{T_m>n\}$.

Next we show (\ref{ix}) with $X_n$ replaced by $\inf_{i\ge n}X_i$.  The proof is
the same as in, for example, \cite{BS08b}, p.\,849. We observe that for all
$m,n,p\in\IN$
\[\{\sup_{i\le n}X_i<m\}\subset\{\inf_{i\ge n}X_i<m\}\subset \{\sup_{i\le
n}X_i<m+p\}\cup\{\inf_{i\ge T_{m+p}}X_i<m\}.\]
The following lemma completes the proof of part (i).
\begin{lemma}
  \label{inf}
  \[\lim_{k\to\infty}\sup_{n\ge 1}P_0\left(\inf_{i\ge
      T_n}X_i<n-k\right)=0.\]
\end{lemma}
\noindent We postpone the proof of this lemma until the end of the section.

\textit{Proof of part (ii).} Let $\delta=4$. Theorem~\ref{progeny}
implies that the distribution of $S^V$ is in the domain of attraction
of the normal distribution (\cite{Fe71}, Chapter~XVII.5). Norming
constants are easily computed to be (see \cite{Fe71}, Chapter~XVII.5,
(5.23) with $C=1$) $\sqrt{C_2n\log n}$. The constant $b$, which
appears in the statement is equal to $C_2/2$. Relations (\ref{iitime})
and (\ref{iix}) follow in the same way as for part (i). \qed

\begin{proof}[Proof of Lemma~\ref{inf}]
Let  $P_{n,k}:=P_0\left(\inf_{i\ge T_n}X_i<n-k\right).$

   \textit{Step 1.}  The supremum over $n\ge 1$ can be reduced to the
  maximum over $n\in\{1,2,\dots,k\}$: $\sup_{n\ge
    1}P_{n,k}=\max_{1\le n\le k}P_{n,k}.$

  Indeed, consider $P_{k,k}$ and $P_{m+k,k}$ for $m\ge 1$. The
  corresponding events $\{\inf_{i\ge T_k}X_i<0\}$ and $\{\inf_{i\ge
    T_{k+m}}X_i<m\}$ depend on the behavior of the process only at
  times when $X_i$ is in $[0,\infty)$ and $[m,\infty)$ respectively.
  But at times $T_0$ and $T_m$ the walk is at $0$ and $m$
  respectively, and the distributions of the environments starting from
  the current point to the right of it are the same under $P_0$. We
  conclude that $P_{k,k}=P_{m+k,k}$.  This is essentially the content
  of Lemma~10 from \cite{Ze05}.  The proof does not use the positivity of
  cookies so it can be applied here.

  \textit{Step 2.} We list four elementary properties of
  $\{P_{n,k}\}$, $n,k\ge 1$.
 \begin{itemize}
\item [(a)] $P_{n,k}\ge P_{n,k+m}$ for all $1\le n\le k$ and $m\ge 0$;
\item [(b)] $P_{n,k+m}$ converges to $0$ as $m\to\infty$ for each
  $k\ge n\ge 1$;
\item [(c)] $P_{n,k}\ge P_{n+m,k+m}$ for all $n\le k$ and $m\ge 0$;
\item [(d)] $P_{n+m,k+m}$ converges to $0$ as $m\to\infty$ for each
  $k\ge n\ge 1$.
\end{itemize}
Inequality (a) is obvious. Part (b) follows from the transience of
$X$. Namely,  $\inf_{i\ge T_n} X_i>-\infty$ a.s.\ but $n-(k+m)\to-\infty$ as
$m\to\infty$. Inequality (c) is also obvious: since $T_n<T_{m+n}$ we
have
\begin{equation*}
  \{\inf_{i\ge
    T_n}X_i<n-k\}
  \supset \{\inf_{i\ge
    T_{n+m}}X_i<n-k\}=\{\inf_{i\ge
    T_{n+m}}X_i<(n+m)-(k+m)\}.
\end{equation*}
The convergence in (d) again follows from the
transience: $X_i\to\infty$ as $i\to\infty$ a.s.\ implies that
$\inf_{i\ge T_{m+n}}X_i\to\infty$ as $m\to\infty$ a.s.\ but $(k+m)-(n+m)$
stays constant.

\textit{Step 3.} Take any $\epsilon>0$ and using (d) choose an $m$ so
that $P_{m,m}<\epsilon$. Properties (a) and (c) imply that
$P_{n,n+i}<\epsilon$ for all $i\ge 0$ and $n\ge m$. Using (b),
for $n=1,2,\dots,m-1$ choose $k_n$ so that $P_{n,k_n}<\epsilon$. Let
$K=\max_{1\le n\le m} k_n$ (naturally, we set $k_m=m$). Then
$P_{n,k}<\epsilon$ for all $n\le k$ and $k\ge K$ that is $\max_{1\le
  n\le k} P_{n,k}<\epsilon$ for all $k\ge K$.
\end{proof}

\begin{remark}\label{alt}
{\em Theorem~\ref{X} and Theorem~\ref{progeny} imply Theorem~\ref{BS} for
general cookie environments satisfying conditions (A1) and (A2). The proof is
the same as in Section 6 of \cite{BS08b}) and uses Lemma~\ref{inf}. }
\end{remark}

%%%%%%%%%%%%%%%%%%%%%%%%%%%%%%%%%%%%%%%%%%%%%%%%%%%%%%%%%%%%%%%%%%%%%%%
\appendix
\section{Proofs of technical results}

We shall need the following simple lemma.
\begin{lemma}
  \label{ldg}
  Let $(\xi_i)_{i\in\IN}$ be i.i.d.\ $\mathrm{Geom}(1/2)$ random
    variables. There exists a constant $c_6>0$ such that for all
    $x,y\in\IN$
\[P\left(\sum_{i=1}^x (\xi_i -1) \ \ge y\right) \le e^{-c_6
    y^2/x} \vee e^{-c_6y}.\]
\end{lemma}

\begin{proof}
  Let $\phi(t)=\log
  Ee^{t(\xi_1-1)}=-t-\log(2-e^t), \
  t\in[0,\log 2)$. Then $\phi'(0)=0$ and there is a constant
  $C>0$ such that $\phi(t)\le Ct^2$ for all
  $t\in[0,(\log2)/2]$. By Chebyshev's inequality, for each
  $t\in[0,(\log 2)/2]$
\[\log P\left(\sum_{i=1}^x (\xi_i -1) \ge y\right) \le
x\phi(t)-yt\le Cxt^2-yt,\] and,
therefore,
\begin{align*}
  \log P\left(\sum_{i=1}^x (\xi_i -1) \ge y\right) &\le
-\max_{t\in[0,(\log 2)/2]} (yt-Cxt^2)\\&\le
-\frac14\min\left\{\frac{y^2}{Cx},y\log 2\right\}.
\end{align*}
This gives the desired inequality with
$c_6=\frac14\min\left\{C^{-1},\log 2\right\}$.
\end{proof}

\begin{proof}[Proof of Lemma~\ref{overshoot}]
  We shall prove part (i). The proof of part (ii) is very similar and
  is omitted. 

  To prove part (i) it is enough to show the existence of $c_1,c_2>0$
  such that (\ref{over}) holds for all $x\ge 2M+1$ and $y\ge 6M$. The
  extension to all $y\ge 0$ is done simply by replacing the constant
  $c_1$ with \[\max\left\{c_1,
    \left(e^{-c_2(6M)^2/(2M+1)}+e^{-6Mc_2}\right)^{-1}\right\},\]
  since the left hand side of (\ref{over}) is at most $1$ and the
  right hand side of (\ref{over}) is increasing in $x$ and decreasing
  in $y$. We have
\begin{align*}
  P_z(&V_{\tau_x}>x+y,\tau_x<\sigma_0)=\sum_{n=1}^\infty
  P_z(V_{\tau_x}>x+y,\tau_x=n,\tau_x< \sigma_0)\\=&\sum_{n=1}^\infty
  P_z(V_n>x+y,V_n\ge x, \tau_x=n,\tau_x<
  \sigma_0)\\=&\sum_{n=1}^\infty \sum_{r=1}^{x-1}P_z(V_n>x+y,V_n\ge x,
  V_{n-1}=r,0<V_j<x,j\in\{1,\dots,n-2\}) \\=& \sum_{n=1}^\infty
  \sum_{r=1}^{x-1}P_z(V_n>x+y\,|\,V_n\ge
  x,V_{n-1}=r)\\&\makebox[2.7cm]{\ }\times P_z(V_n\ge x,
  V_{n-1}=r,0<V_j<x,j\in\{1,\dots,n-2\})\\=&\sum_{n=1}^\infty
  \sum_{r=1}^{x-1}\frac{P_z(V_n>x+y\,|\,V_{n-1}=r)}{P_z(V_n\ge
    x\,|\,V_{n-1}=r)}P_z(\tau_x<\sigma_0,\tau_x=n,
  V_{n-1}=r)\\
% \end{align*}
% \begin{align*}
=&\sum_{n=1}^\infty
  \sum_{r=1}^{x-1}\frac{P_r(V_1>x+y)}{P_r(V_1\ge
    x)}P_z(\tau_x<\sigma_0,\tau_x=n, V_{n-1}=r)\\\le& \max_{0\le
    r<x}\frac{P_r(V_1>x+y)}{P_r(V_1\ge x)}\sum_{n=1}^\infty
  \sum_{r=1}^{x-1}P_z(\tau_x<\sigma_0,\tau_x=n,
  V_{n-1}=r)\\=&\max_{0\le r<x}\frac{P_r(V_1>x+y)}{P_r(V_1\ge
    x)}P_z(\tau_x<\sigma_0).
\end{align*}
Thus, \[\max_{0\le z<x} P_z(V_{\tau_x}>x+y\,|\,\tau_x<\sigma_0)\le
\max_{0\le z<x}\frac{P_z(V_1>x+y)}{P_z(V_1\ge x)}.\] To estimate the
last ratio we recall that $V_1$ is the sum of the number of offspring
produced by each of $z$ particles and by the immigrant particle. The
offspring distribution of at most $M$ particles can be affected by the
cookies. For notational convenience we shall use $\xi_j$,
$j=1,2,\dots,z+1$, to denote the number of offspring of the $j$-th
particle. All that was said in Section 2 about $\zeta^{(k)}_m$ applies
now to to $\xi_m$, $m=1,2,\dots,M$, and $\xi_m$, $m>M$, are just
i.i.d.\ $\mathrm{Geom}(1/2)$ random variables. Abbreviate $\xi_j-1$ by
$\xi'_j$, $j\in\IN$. With this notation,
$V_1=z+1+\sum_{m=1}^{z+1}\xi'_m $. Let ${\cal
  G}_0=\{\emptyset,\Omega\}$, ${\cal G}_n$, $n\in\IN$, be the
$\sigma$-algebra generated by $(\xi_m)_{m\le n}$,
and \[N_x=\inf\left\{n\ge 1\,\Big|\,z+1+\sum_{m=1}^n\xi'_m \ge
  x\right\}.\] Then $\{V_1>x+y\}\subset\{V_1\ge x\}\subset\{N_x\le
z+1\}$ and
\begin{equation*}
  \frac{P_z(V_1>x+y)}{P_z(V_1\ge x)}= \frac{P_z(V_1>x+y)} {E_z(\I_{\{N_x\le
      z+1\}}P_z(V_1\ge x\,|\,{\cal G}_{N_x}))}
\end{equation*}
The proof of part (i) will be complete as soon as
we show that
\begin{itemize}
\item [(a)] the is a constant $c_7>0$ such that $P_z(V_1\ge x\,|\,{\cal
    G}_{N_x})\ge c_7$ for all $0\le z<x$ and $x\ge 2M+1$;
\item [(b)]there is a constant $c_8>0$ such that for
all $x\ge 2M+1$ and $y\ge 6M$
  \begin{equation}
    \label{nx}
    P_z(V_1>x+y\,|\,N_x\le z+1)\le \left(e^{-c_8y}\vee
  e^{-c_8y^2/x}\right).
  \end{equation}
\end{itemize}

Proof of (a). Consider two cases: (i) $z\ge M$ and (ii) $x-z\ge
M+1$. In case (i) at least one of $\xi_m$, $1\le m\le
z+1$, has a $\mathrm{Geom}(1/2)$ distribution. The distribution of 
each $\xi_m$ is supported on the non-negative integers, and,
trivially, $\sum\limits_{m=1}^M \xi'_m \ge -M$. Therefore, on the
event $\{N_x\le z+1\}$
\begin{multline}
  \label{M}
 P_z(V_1\ge x\,|\, \mathcal{G}_{N_x}) \ge
 P_z\left(\sum_{m=N_x+1}^{z+1}\xi'_m \ge 0 \,\Big|\,
 \mathcal{G}_{N_x}\right)%\\
\ge \min_{n\ge
   M+1}P\left(\sum_{m=M+1}^n\xi'_m \ge M\right)\ge c_7>0,
\end{multline}
since for all $n\ge M+1$ the probabilities under the minimum sign are
strictly positive and by the Central Limit Theorem \[\lim_{n\to\infty}
P\left(\sum_{m=M+1}^n\xi'_m \ge M\right)=\frac12.\] We recall that
when $N_x=z+1$ the first sum in (\ref{M}) is empty and the probability
above is equal to 1, so the lower bond holds for this case as well.

Case (ii) is even simpler. We have $x-z\ge M+1$. This implies that all
cookies at site $0$ have already been used for offspring of the first
$N_x$ particles. Therefore, the sum $\sum_{m=N_x+1}^{z+1}\xi'_m $ (when
it is not zero) is equal to the non-trivial sum of centered
$\mathrm{Geom}(1/2)$ random variables. Therefore, in case (ii) we
again get (\ref{M}) where the rightmost $M$ is replaced by $0$.

Proof of (b). Observe that 
\begin{equation*}
  \{V_1>x+y\}\subset
\Big\{\sum_{m=N_x+1}^{z+1}\xi'_m \ge
y/2\Big\}\cup\Big\{z+1+\sum_{m=1}^{N_x}\xi'_m -x\ge y/2\Big\}.
\end{equation*}
Using the assumption $y\ge 6M$ and Lemma~\ref{ldg} we get
\begin{align*}
  P_z\Big(\sum_{m=N_x+1}^{z+1}\xi'_m \ge y/2,N_x\le
  z+1\Big)&= E_z\Big(\I_{\{N_x\le
    z+1\}}P_z\Big(\sum_{m=N_x+1}^{z+1}\xi'_m \ge y/2\,\Big|\, {\cal
    G}_{N_x}\Big)\Big)\\\le\max_{1\le n\le
    x}P_z\Big(\sum_{m=M+1}^{M+n}\xi'_m &\ge y/2-M\Big)P_z(N_x\le
  z+1)\\\le\max_{1\le n\le
    x}P_z\Big(\sum_{m=M+1}^{M+n}\xi'_m &\ge y/3\Big)P_z(N_x\le
  z+1)\\&\le \left(e^{-c_6y^2/(9x)}\vee
    e^{-c_6y/3}\right)P_z(N_x\le z+1).
\end{align*}
Finally, we estimate the probability of the second set:
\begin{align*}
  &P_z\Big(z+1+\sum_{m=1}^{N_x}\xi'_m -x\ge y/2,N_x\le z+1\Big)
  \\&=\sum_{n=1}^{z+1}P_z\Big(\sum_{m=1}^{N_x}\xi'_m \ge
  y/2+x-z-1,N_x=n\Big)\\&=\sum_{n=1}^{z+1}\sum_{\ell=1-n}^{x-z-2}
  P_z\Big(\sum_{m=1}^n\xi'_m \ge
  y/2+x-z-1,N_x=n,\sum_{m=1}^{n-1}\xi'_m =\ell\Big)\\
% \end{align*}
% \begin{align*}
  &=\sum_{n=1}^{z+1}\sum_{\ell=1-n}^{x-z-2}
  P_z\Big(\xi'_n \ge
  y/2+x-z-1-\ell\,\Big|\,N_x=n,\sum_{m=1}^{n-1}\xi'_m
  =\ell\Big)\\&\makebox[7.3cm]{\ } \times
  P_z\Big(N_x=n,\sum_{m=1}^{n-1}\xi'_m =\ell\Big)
  \\&=\sum_{n=1}^{z+1}\sum_{\ell=1-n}^{x-z-2} P_z\Big(\xi_n\ge
  y/2+x-z-\ell\,\Big|\,\xi_n\ge x-z-\ell,N_x=n,\sum_{m=1}^{n-1}\xi'_m
  =\ell\Big)\\&\makebox[7.3cm]{\ } \times
  P_z\Big(N_x=n,\sum_{m=1}^{n-1}\xi'_m
  =\ell\Big)\\&\le P_z(\xi_{M+1}\ge y/2-M)\sum_{n=1}^{z+1}\sum_{\ell=1-n}^{x-z-2}
  P_z\Big(N_x=n,\sum_{m=1}^{n-1}\xi'_m
  =\ell\Big)\\&\makebox[7.8cm]{\ } \le 2^{-y/3}P_z(N_x\le z+1).
\end{align*}
This finishes the proof.
\end{proof}

\begin{proof}[Proof of part (ii) of Lemma~\ref{scale}]
Let $s\in C^\infty_0([0,\infty))$ be a non-negative 
function such that $s(x)=x^{\delta}$ on $(2/(3a),3a/2)$.
Fix an $n$ such that $a^{n-1}>M$ and define the process $U^n:=(U^n_k)_{k\ge 0}$
by
\[U^n_k = s\left(\frac{V_{k \wedge \gamma}}{a^n}\right).\] We shall show that
when $n$ is large $U^n$ is close
to being a martingale (with respect to its natural filtration
$(\mathcal{F}_k)_{k\ge 0}$). $U^n$ is just a discrete 
version of the martingale used in the proof of Lemma~\ref{elem}.

On the event $\{\gamma>k\}$ we have
\[E(U^n_{k+1}\,|\, \mathcal{F}_k) =
E\left(s\left(\frac{V_{k+1}}{a^{n}}\right)\,\Big|\, \mathcal{F}_k\right) =
E \left(s\left(\frac{V_{k}}{a^n} + \frac{V_{k+1} -
      V_{k}}{a^n}\right)\, \Big|\, \mathcal{F}_k\right),\] and
\begin{multline*}
E(U^n_{k+1}\,\vert\, \mathcal{F}_k) - U^n_{k}= \\E \left[s'
\left(\frac{V_{k}}{a^{n}}\right) \frac{V_{k+1} - V_{k}}{a^{n}}\,\Big\vert\,
\mathcal{F}_k
\right] + \frac{1}{2} E \left[s'' \left(\frac{V_{k}}{a^{n}}\right)
\frac{{(V_{k+1}
    -V_{k})}^{2}}{a^{2n}} \,\Big\vert\, \mathcal{F}_k\right] + r^n_k=\\
    -\frac{\delta -1}{a^{n}}\, s' \left(\frac{V_{k}}{a^n}\right)+\frac{1}{2} E
\left[s'' \left(\frac{V_{k}}{a^{n}}\right) \frac{{(V_{k+1}
    -V_{k})}^{2}}{a^{2n}} \,\Big\vert\, \mathcal{F}_k\right] + r^n_k,
\end{multline*}
where $r^n_k$ is the error, which we shall estimate later.
By (\ref{diffV}), the second term on the right hand side of the above equality
is equal to 
\[\frac{1}{2a^{2n}}\,s'' \left(\frac{V_{k}}{a^{n}}\right) E
\left[\left(1+\sum_{m=1}^M \zeta^{(k) \prime}_m + \sum_{m=1}^{V_{k} - M+1}
\xi^{(k)\prime}_{m}\right)^2\Big\vert\,\mathcal{F}_k\right],\]
where $\zeta_m^{(k)\prime}=\zeta_m^{(k)}-1$ and
$\xi_m^{(k)\prime}=\xi_m^{(k)}-1$.
Since $\sum_{m=1}^M \zeta_m^{(k) \prime}$ is independent from all 
$\xi_{m}^{(k)\prime}$, $m\ge 1$, and $V_k$, the last formula reduces to
\[\frac{1}{2a^{2n}} \, s'' \left( \frac{V_{k}}{a^n}\right)\left( 2(V_{k}-M+1) +
E\left[\left(1+\sum_{m=1}^M \zeta^{(k) \prime}_m\right)^{2}\right]\right).\]
Using the fact that $xs''(x)+(1-\delta)s'(x)=0$ for $x\in(2/(3a),3a/2)$ we get
that on the event $\{\gamma>k\}$
\begin{equation*}
  E(U^n_{k+1}\,\vert\, \mathcal{F}_k) - U^n_{k}=
  \frac{1}{a^{2n}} \, s'' \left( \frac{V_{k}}{a^n}\right)\left( 1-M+
    \frac12E\left[\left(1+\sum_{m=1}^M \zeta^{(k)
          \prime}_m\right)^{2}\right]
  \right)+r^n_k.
\end{equation*}
The first term on the right hand side is bounded 
in absolute value by $K_1/a^{2n}$ for some constant $K_1$. Thus it remains to
estimate $r^n_{k}$.  By Taylor's expansion $r^n_k$ is bounded by
\begin{equation*}
\frac{1}{6} \,\|s'''\|_\infty E\left[\left(\frac{|V_{k+1} -
V_{k}|}{a^n}\right)^3 \Big\vert \,\mathcal{F}_k\right] %\\ 
\le \frac{1}{6}\,
\|s'''\|_\infty  \left( E \left[\left(\frac{V_{k+1} - V_{k}}{a^n}\right)^{4}
\Big\vert\, \mathcal{F}_k\right] \right)^{3/4}.
\end{equation*}
Writing again the difference  $V_{k+1}-V_k$ in terms of geometric random
variables, using independence of $\sum_{m=1}^M \zeta_m^{(k) \prime}$ from all 
$\xi_{m}^{(k)\prime}$, $m\ge 1$, and the fact that $V_k<a^{n+1}$ on
$\{\gamma>k\}$ we find that 
\[E \left[\left(\frac{V_{k+1} - V_{k}}{a^n}\right)^{4} \Big\vert\,
  \mathcal{F}_k\right]\le \frac{K_2}{a^{2n}},\] and, therefore,
$|r^n_k|\le K_3/a^{3n/2}$. Let $R^n_0=0$ and for $k\ge 1$
set \[R^n_k=\sum_{j=1}^{k\wedge \gamma} \left[\frac{1}{a^{2n}} \, s''
  \left( \frac{V_{j}}{a^n}\right)\left( 1-M+
    \frac12E\left[\left(1+\sum_{m=1}^M \zeta^{(j)
          \prime}_m\right)^{2}\right]\right)+r^n_j\right].\] Then
$U^n_k-R^n_k$ is a martingale with the initial value $U^n_0$. Our
bounds on the increments of the process $(R^n_k)_{k\ge 0}$ and
Proposition~\ref{time} below imply that
\[E|R^n_\gamma| \le \frac{K_4}{a^{3n/2}}\, E\gamma\le \frac{K_5}{a^{n/2}}. \] 
This allows us to pass to the limit as $k\to\infty$ and conclude that
$U^n_0 =  EU^n_\gamma  - ER^n_\gamma$.
Thus,
\begin{align*}
U^n_0 -\frac{K_5}{a^{n/2}} & \le EU^n_\gamma \le 
P\left(V_\gamma \in [a^{n+1},a^{n+1}+ a^{2(n-1)/3} )\right ) s(a+a^{-(n+2)/3}
)\\& + P\left(V_\gamma \in (a^{n-1} -  a^{2(n-1)/3},a^{n-1}  ] \right) s(a^{-1}
) + E\left(U^n_\gamma \I_{\{d(V_\gamma, (a^{n-1},a^{n+1})) \ge  a^{2(n-1)/3}\}}
\right).
\end{align*}
By part (i), we obtain that
\[
P(V_\gamma \geq a^{n+1})\, a ^\delta + P(V_\gamma \leq a^{n-1})\, a ^{-\delta} \geq
U^n_0 - K_6/a^{n/3}.
\]
Similarly we get
\[
P(V_\gamma \geq a^{n+1})\, a ^\delta + P(V_\gamma \leq a^{n-1})\, a ^{-\delta} \leq
U^n_0 +K_7/a^{n/3}.
\]
This completes the proof.
\end{proof}

\begin{proposition} \label{time}
There exists $C_8 \in (0, \infty) $ so that for all $x > 0$, 
\[
\sup _{x \leq y \leq 2x} E^V_y\Bigg( \sum_{r=0} ^{\sigma^V_{x/2}} \I_{V_r \in
[x,2x]}\Bigg) < C_8x.
\]
\end{proposition}
\begin{proof}
By the usual compactness considerations and Lemma~\ref{weakcon}, there exists
$c>0$ such that $P^V_y(\sigma^V_{x/2} < x)>c$ for all $x>0$ and $y\in[x,2x]$.
From this and the Markov property applied to successive re-entries to the
interval $[x,2x]$ (see the proof of Proposition~\ref{reentry} for details), we
obtain 
\[
P^V_y\Bigg( \sum_{r=0}^{\sigma^V_{x/2}} \I_{V_r \in [x,2x]} > nx\Bigg) \ \leq \
(1-c)^n,
\]
and the result follows.
\end{proof}

\noindent \textbf{Acknowledgments.}  The authors are grateful to the
Institute of Henri Poincar\'e, where this project was started, for
hospitality. E.~Kosygina also thanks the Institute for Advanced Study
for a stimulating research environment. Her work was partially supported
by NSF grants DMS-0825081 and DMS-0635607.  T.~Mountford was partially
supported
by the C.N.R.S.\ and by the SNSF grant 200020-115964.

%%%%%%%%%%%%%%%%%%%%%%%%%%%%%%%%%%%%%%%%%%%%%%%%%%%%%%%%%%%%%%%%%%%%%%%%
\bibliographystyle{plain}

\begin{thebibliography}{99}

\bibitem{BS08a}{\sc A.-L.\ Basdevant} and {\sc A.\ Singh}\ (2008).
On the speed of a cookie random walk. 
\textit{Probab.\ Theory Related Fields} {\bf 141}, no.\ 3-4, 625 -- 645.

\bibitem{BS08b}{\sc A.-L.\ Basdevant} and {\sc A.\ Singh}\
  (2008).  Rate of growth of a transient cookie random walk.
  \textit{Electron.\ J.\ Probab.}  {\bf 13}, no.\ 26, 811 -- 851.

\bibitem{BW03}{\sc I.\ Benjamini} and {\sc D.B.\ Wilson}\ (2003).
Excited random walk. \textit{Electron.\ Comm.\ Probab.} {\bf 8}, 86 -- 92

\bibitem{BR07}{\sc J.\ B\'erard} and {\sc A.\ Ram\'irez}\ (2007).
 Central limit theorem for the excited random walk in dimension $d \geq 2$. 
\textit{Elect.\ Comm.\ in Probab.}  \textbf{12}, no.\ 30, 303 -- 314

\bibitem{CD99}{\sc L.\ Chaumont} and {\sc R.A.\ Doney}\ (1999).
  Pathwise uniqueness for perturbed versions of Brownian motion and
  reflected Brownian motion. \textit{Probab.\ Theory Related Fields}
  {\bf 113}, no.\ 4, 519 -- 534.

\bibitem{Da99}{\sc B.\ Davis}\ (1999). Brownian motion and random
  walk perturbed at extrema.  \textit{Probab. Theory Related Fields}
  {\bf 113} no.\ 4, 501--518.

\bibitem{Do08} {\sc D.\ Dolgopyat} (2008). Central limit theorem
for excited random walk in the recurrent regime. 7 p.. Preprint.

\bibitem{Dur96}{\sc R.\ Durrett}\ (1996). \textit{Probability:
    theory and examples.} 2nd edition. Duxbury Press, Belmont, CA, xiii+503 pp.

\bibitem{Gu88}{\sc A. Gut}\ (1988). \textit{Stopped random walks. Limit
theorems and applications.} Applied Probability. A Series of the Applied
Probability Trust, 5. Springer-Verlag, New York, x+199 pp.

\bibitem{EK86}{\sc S. Ethier} and {\sc T. Kurtz}\ (1986).
  \textit{Markov processes.} John Wiley \& Sons, Inc., New York, x+534 pp.

\bibitem{Fe71}{\sc W. Feller} (1971). \textit{
An introduction to probability theory and its applications.} Vol. II.
2nd edition. John Wiley \& Sons, Inc., New York-London-Sydney, xxiv+669 pp. 

\bibitem{FYK90}{\sc S.K.\ Formanov, M.T.\ Yasin} and
  {\sc S.V.\ Kaverin} (1990). Life spans of Galton-Watson processes
  with migration. (Russian) In \textit{Asymptotic problems in
    probability theory and mathematical statistics,} edited by T.A.\
  Azlarov and Sh.K.\ Formanov. Tashkent: Fan, 176 p. (pp.\ 117--135)

\bibitem{HH08}{\sc R.\ van der Hofstad} and {\sc M.\
    Holmes} (2008). Monotonicity for excited random walk in high 
dimensions. \verb+arXiv:0803.1881v2 [math.PR]+
  
\bibitem{KKS75}{\sc H.\ Kesten, M.V.\ Kozlov} and {\sc
    F.\ Spitzer}\ (1975). A limit law for random walk in random 
environment.
  \textit{Compositio Math.} \textbf{30}, 145--168

\bibitem{KZ08}{\sc E.\ Kosygina} and {\sc M.P.W. Zerner}. Positively and
  negatively excited random walks, with branching processes (2008).
  \textit{Electron. J. Probab.} {\bf 13}, no.\ 64, 1952 -- 1979.

\bibitem{K73} {\sc M.V.\ Kozlov}. Random walk in a one-dimensional random
medium (1973). \textit{Teor. Verojatnost. i Primenen.} {\bf 18}, 406 -- 408.

\bibitem{Ko03}{\sc G.\ Kozma}\ (2003).
Excited random walk in three dimensions has positive speed. Preprint.
\verb+arXiv:math/0310305v1 [math.PR]+

\bibitem{MPV06}{\sc T.\ Mountford, L.P.R.\ Pimentel} and {\sc
    G.\ Valle}\ (2006).  On the speed of the one-dimensional excited
  random walk in the transient regime.  \textit{Alea} \textbf{2}, 279--296
		
\bibitem{Ze05}{\sc M.P.W.\ Zerner}\ (2005).
Multi-excited random walks on integers. 
 \textit{Probab.\ Theory Related Fields}  \textbf{133}, 98 -- 122

\bibitem{Ze06}{\sc M.P.W.\ Zerner}\ (2006).
Recurrence and transience of excited random walks on $\mathbb{Z}^d$ and strips.
\textit{Electron.\ Comm.\ Probab.}  \textbf{11}, no.\ 12, 118 -- 128


\end{thebibliography}

\bigskip

{\sc \small
\begin{tabular}{ll}
Department of Mathematics& \hspace*{17mm}Ecole Polytechnique F\'ed\'erale\\
Baruch College, Box B6-230& \hspace*{17mm}de Lausanne\\
One Bernard Baruch Way&\hspace*{17mm}D\'epartement de math\'ematiques\\
New York, NY 10010, USA&\hspace*{17mm}1015 Lausanne, Switzerland\\
{\verb+elena.kosygina@baruch.cuny.edu+}& \hspace*{17mm}{\verb+thomas.mountford@epfl.ch+}
\end{tabular}\vspace*{2mm}

}

\end{document}